\documentclass[12pt,a4paper]{article}
\usepackage{amsmath,amssymb,amsthm,amsfonts}
\usepackage{graphicx}
\usepackage{xcolor}
\usepackage{hyperref}
\usepackage{mathtools}
\usepackage{enumitem}
\usepackage{booktabs}
\usepackage{caption}
\usepackage{subcaption}
\usepackage[margin=2.5cm]{geometry}
\hypersetup{colorlinks=true,linkcolor=blue!50!black,citecolor=blue!50!black,urlcolor=blue!50!black}
\graphicspath{{figures/}}

\newtheorem{theorem}{Theorem}[section]
\newtheorem{lemma}[theorem]{Lemma}
\newtheorem{proposition}[theorem]{Proposition}

\newtheorem{introtheorem}{Theorem}

\theoremstyle{definition}
\newtheorem{definition}[theorem]{Definition}
\newtheorem{remark}[theorem]{Remark}
\newtheorem{problem}[theorem]{Open Problem}

\newcommand{\R}{\mathbb{R}}
\newcommand{\C}{\mathbb{C}}
\newcommand{\Z}{\mathbb{Z}}
\newcommand{\x}{\mathbf{x}}
\newcommand{\n}{\mathbf{n}}

\DeclareMathOperator{\Ree}{Re}
\DeclareMathOperator{\Imm}{Im}
\newcommand{\viewer}[1]{\href{https://surfaces.igorkozyrev.com/\##1}{\nolinkurl{surfaces.igorkozyrev.com/\##1}}}

\title{The Costa Surface with a Free-Boundary Hole:\\ an Embedded Exterior Free-Boundary Minimal Surface\\ with Three Ends}
\author{Igor Kozyrev\\[4pt]
{\normalsize Independent University of Moscow}\\
{\normalsize Bolshoy Vlasyevskiy per.~11, Moscow 119002, Russia}\\
{\normalsize\texttt{ikozyrev@gmail.com}}}
\date{October 2026}

\begin{document}
\maketitle

\begin{abstract}
An exterior free-boundary minimal surface is a properly immersed minimal surface outside the unit
ball whose compact boundary lies on the unit sphere and meets it orthogonally. By Mazet and
Mendes, in $\R^3$ every such surface with one regular end is a catenoidal piece, and every stable
one with parallel embedded regular ends is catenoidal. We construct connected embedded examples with three
ends from the Costa surface $M$, placed so that its central saddle point is at the origin. For
every sufficiently small $\varepsilon>0$ we perturb $\varepsilon^{-1}(M\setminus B_\varepsilon)$ into a properly
embedded exterior free-boundary minimal surface of genus one with one boundary circle and three
ends, two catenoidal and one planar. The construction is a small-hole perturbation. The only error of
the initial surface is in the boundary condition. The symmetry group of the Costa surface removes
all obstructions: at the hole it forbids the logarithmic and the tilting modes, and
globally it kills the bounded Jacobi fields of the Costa surface, which is nondegenerate by
Nayatani. In the conformal compactification the Jacobi operator is smooth also at the ends, so the
linear analysis takes place on a compact torus with one small hole. The same construction applied to the
Costa--Hoffman--Meeks surfaces gives such examples of every genus $k\geq1$, of index at least $2k+3$. We also report numerical experiments with two
independent implementations of the Costa surface. They confirm the period condition, the symmetries, the boundary
error, the index and nullity of the Jacobi operator and the absence of an equivariant kernel. They also compute the
free-boundary correction near the hole, which gives pictures of the surfaces.
\end{abstract}

\section{Introduction}\label{sec:intro}

\begin{definition}\label{def:exterior}
An \emph{exterior free-boundary minimal surface} (exterior FBMS) in $\R^3$ is a smooth surface
$\Sigma$ with compact boundary together with a proper minimal immersion
$\x\colon\Sigma\to\R^3\setminus B^3$ such that $\x^{-1}(S^2)=\partial\Sigma$, $\x$ meets $S^2$ orthogonally
along $\partial\Sigma$, and every component of $\Sigma$ has non-empty boundary. Here $B^3$ is the open unit
ball.
\end{definition}

The catenoidal pieces $C_\alpha$ and the exterior of the equatorial disk are the rotationally symmetric
examples. Mazet and Mendes~\cite{MazetMendes2022} proved that an exterior FBMS in $\R^3$ with one
regular end is one of them~\cite[Theorem~1.2]{MazetMendes2022}, and that a stable exterior FBMS with parallel
embedded regular ends is catenoidal~\cite[Theorem~1.1]{MazetMendes2022}. Disconnected embedded examples with several
ends are trivial: the part of a critical catenoid outside the ball consists of two catenoidal pieces. In~\cite{KozyrevRealization}
we constructed by gluing connected embedded examples with $\ell\geq2$ ends, $b\geq1$ boundary components and every genus; all their
ends are catenoidal. Here we give a different construction, a perturbation of a complete minimal surface with a small free-boundary
hole, which produces connected embedded examples with a planar end.

\begin{introtheorem}\label{thm:A}
Let $M$ be the Costa surface in the position of Section~\ref{sec:costa}, with its saddle point at the
origin, and let $G$ be the group of order $8$ of its symmetries generated by $\sigma_1,\sigma_2,\tau$ (Section~\ref{sec:costa}).
There is $\varepsilon_0>0$ such that for every $\varepsilon\in(0,\varepsilon_0)$ there is a connected properly embedded exterior FBMS $\Sigma_\varepsilon\subset\R^3\setminus B^3$ with the
following properties.
\begin{enumerate}[label=(\roman*)]
\item $\Sigma_\varepsilon$ has genus one, one boundary circle and three ends: two catenoidal ends and one
planar end, all of them embedded, of multiplicity one and parallel.
\item $\Sigma_\varepsilon$ is invariant under $G$.
\item $\varepsilon\Sigma_\varepsilon=\{\Phi^{u_\varepsilon(p)/\theta(p)}(\x(p)):\ p\in\Omega_\varepsilon\setminus\{0,\frac12,\frac i2\}\}$ for the flow $\Phi$ and the function $\theta$ of
Section~\ref{sec:perturb} and a function $u_\varepsilon$ with $|u_\varepsilon(p)|\leq C\varepsilon\,\min(|\x(p)|,1)$, and $\varepsilon\Sigma_\varepsilon\to M$ smoothly on compact subsets of
$\R^3\setminus\{0\}$ as $\varepsilon\to0$.
\item $\Sigma_\varepsilon$ is unstable.
\end{enumerate}
\end{introtheorem}

Part (iv) follows from~\cite[Theorem~1.1]{MazetMendes2022}, since $\Sigma_\varepsilon$ is not catenoidal (see the end of the
proof for its hypotheses). The surfaces of~\cite{KozyrevRealization} have catenoidal ends only, so $\Sigma_\varepsilon$ is not among them.

\begin{figure}[ht]
\centering
\begin{subfigure}{0.36\textwidth}\includegraphics[width=\textwidth]{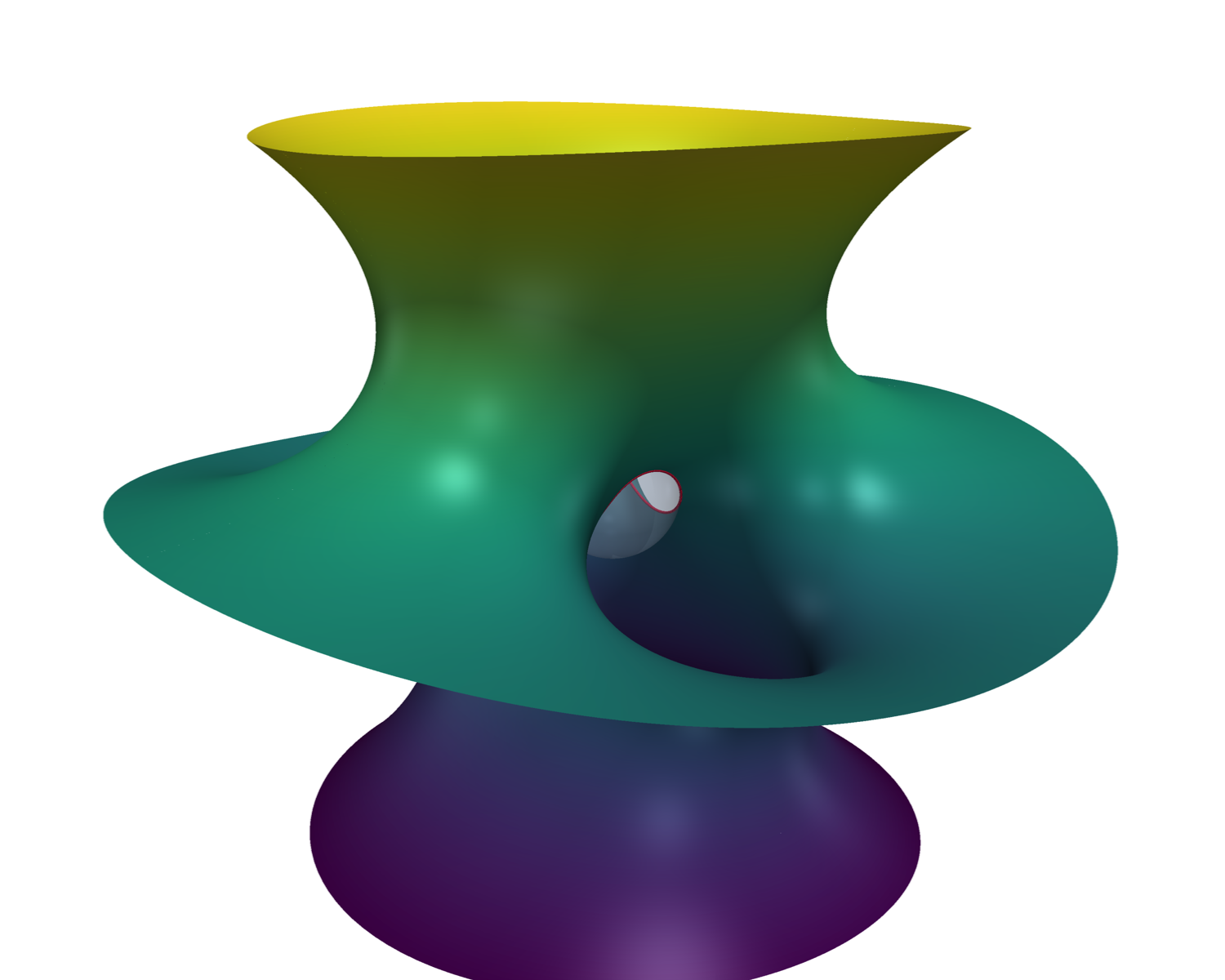}\caption{the surface}\end{subfigure}
\begin{subfigure}{0.31\textwidth}\includegraphics[width=\textwidth]{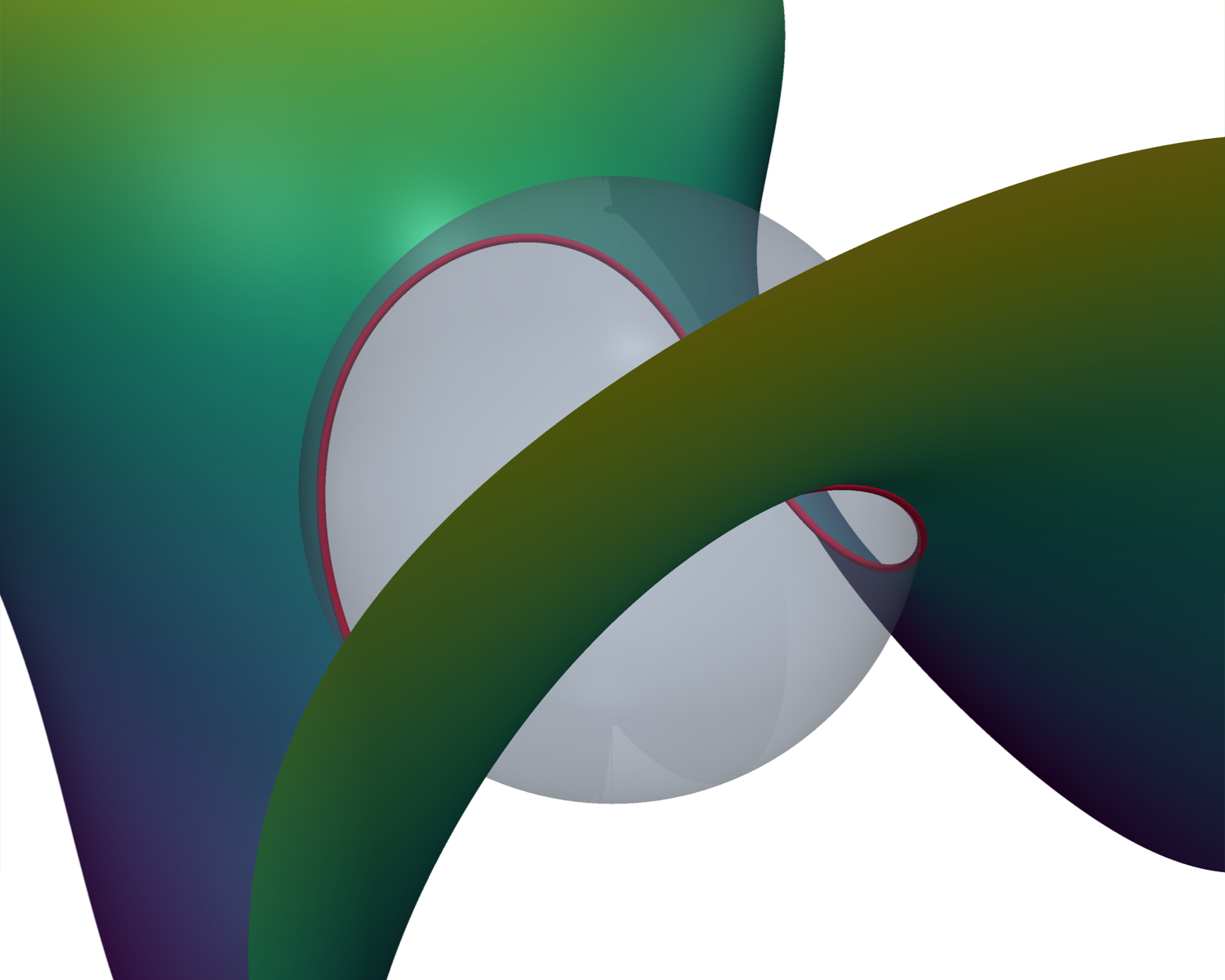}\caption{near the unit sphere}\end{subfigure}
\begin{subfigure}{0.31\textwidth}\includegraphics[width=\textwidth]{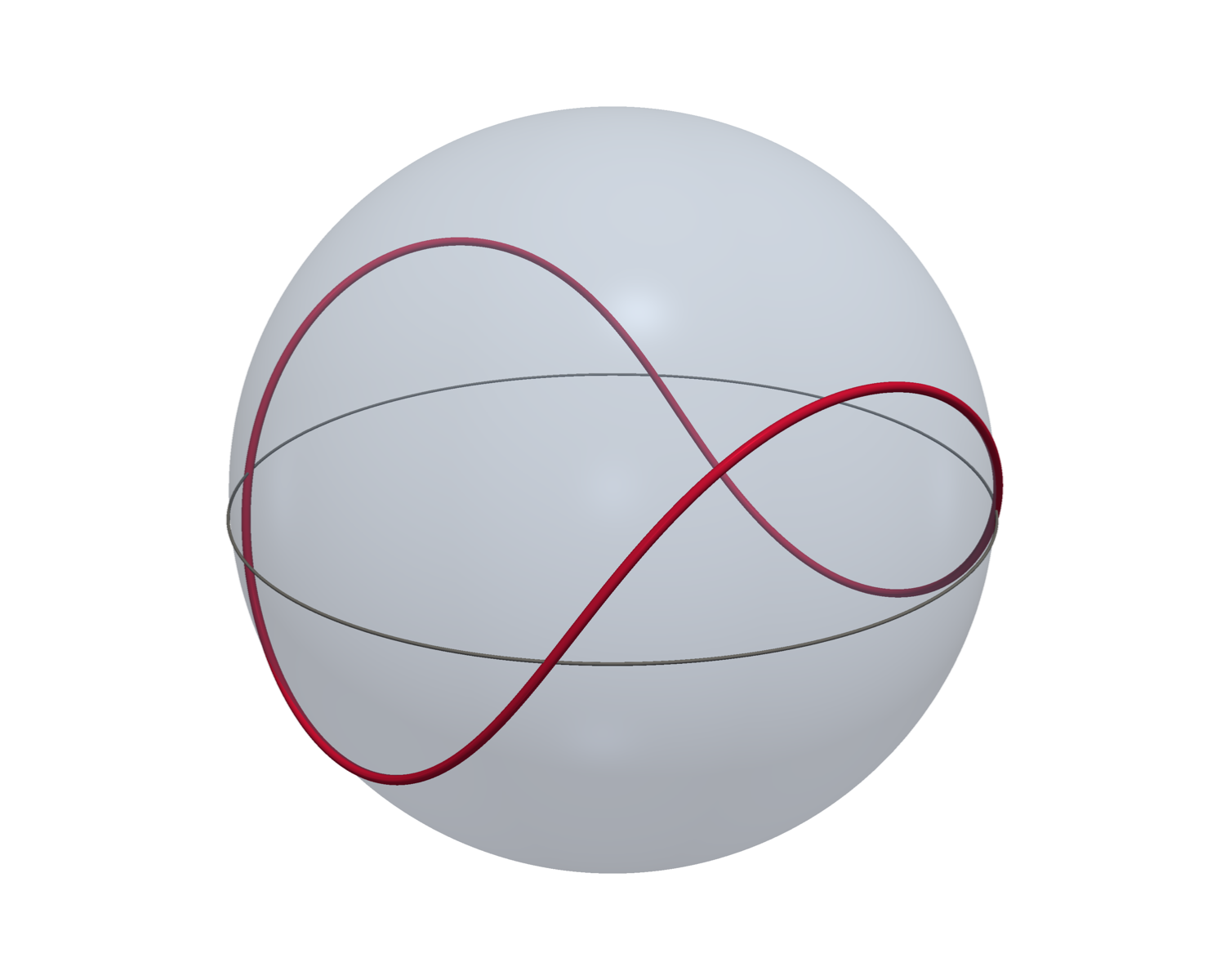}\caption{the boundary curve}\end{subfigure}
\caption{A numerical model based on a local free-boundary collar at $\varepsilon=0.45$ (Section~\ref{sec:num_pictures}).
(a) The Costa surface scaled by $\varepsilon^{-1}$, with the unit ball removed around its saddle point; near the hole it is replaced
by a computed free-boundary minimal collar that meets the unit sphere orthogonally, and farther out it carries a small
correction that is not minimal. The two catenoidal ends point up and down, and the planar end is horizontal. (b) The same near the sphere; the red curve is the free boundary. (c) The
boundary curve alone, with the equator in grey. It reaches latitude $\pm30^\circ$ and crosses the equator at the four points
where the two straight lines of the Costa surface meet the sphere. Theorem~\ref{thm:A} gives $\Sigma_\varepsilon$ for all sufficiently
small $\varepsilon$; for $\varepsilon=0.45$ its existence is not proved, and the distance from the model to it is not estimated. An interactive
three-dimensional version is at \viewer{costa}.}
\label{fig:intro}
\end{figure}

\subsection*{Idea of the proof}
Near its saddle point the Costa surface is the graph of $\frac\kappa2(x_1^2-x_2^2)+O(|x|^4)$, so the sphere
of radius $\varepsilon$ cuts it along an almost equatorial curve, and $M\setminus B_\varepsilon$ fails to be orthogonal to
the sphere only by an angle $\frac\kappa2\varepsilon\cos2\varphi+O(\varepsilon^3)$. The interior equation holds exactly. We
correct the boundary angle by a small perturbation whose boundary stays on the sphere. The linearized
problem is the Jacobi operator of $M$ with a Robin condition on a hole of radius $\varepsilon$. Two kinds of
obstructions could appear: at the hole, harmonic functions of the exterior of a disk of Fourier modes
$n=0,1$ (logarithm and tilts), and globally the bounded Jacobi fields of $M$. The symmetry group removes
both: $G$-equivariant functions have only the modes $n\equiv2\pmod4$ at the saddle point, and no bounded
Jacobi field of the Costa surface is $G$-equivariant. The Costa surface is nondegenerate by
Nayatani~\cite{Nayatani1993}. The linear theory is uniform in $\varepsilon$ because all limiting problems are injective,
and a contraction argument finishes the proof.

The same construction works for the Costa--Hoffman--Meeks surfaces $M_k$ of every genus $k\geq1$~\cite{HoffmanMeeks1990}, which
are nondegenerate by Nayatani and Morabito~\cite{Nayatani1992,Nayatani1993,Morabito2009}. At their central point the
symmetries allow only the Fourier modes $(2j+1)(k+1)$, and the boundary error is of order $\varepsilon^k$.

\begin{introtheorem}\label{thm:B}
Let $k\geq1$, let $M_k$ be the Costa--Hoffman--Meeks surface of genus $k$ in the position of Section~\ref{sec:chm}, with its central
point at the origin, and let $G_k$ be its symmetry group of order $4(k+1)$ described there. There is $\varepsilon_k>0$ such that for
every $\varepsilon\in(0,\varepsilon_k)$ there is a connected properly embedded exterior FBMS $\Sigma_{k,\varepsilon}\subset\R^3\setminus B^3$ with the following
properties.
\begin{enumerate}[label=(\roman*)]
\item $\Sigma_{k,\varepsilon}$ has genus $k$, one boundary circle and three ends: two catenoidal ends and one planar end, all of them
embedded, of multiplicity one and parallel.
\item $\Sigma_{k,\varepsilon}$ is invariant under $G_k$.
\item $\varepsilon\Sigma_{k,\varepsilon}$ is obtained from $M_k\setminus B_\varepsilon$ as in Theorem~\ref{thm:A}(iii), with a displacement of norm at most $C_k\varepsilon^k$, and
$\varepsilon\Sigma_{k,\varepsilon}\to M_k$ smoothly on compact subsets of $\R^3\setminus\{0\}$.
\item The Morse index of $\Sigma_{k,\varepsilon}$ is at least $2k+3$; in particular $\Sigma_{k,\varepsilon}$ is unstable.
\end{enumerate}
For $k=1$, $M_1$ is the Costa surface up to a homothety, $G_1=G$, and Theorem~\ref{thm:B} is Theorem~\ref{thm:A} together with
Proposition~\ref{prop:index}.
\end{introtheorem}

\subsection*{Proofs and computations}
Sections~\ref{sec:costa}--\ref{sec:chm} contain the complete proofs of Theorems~\ref{thm:A} and~\ref{thm:B}, and they use no
numerical input. Section~\ref{sec:numerics} describes numerical experiments of two kinds. The first kind checks the
quantitative ingredients of the proof: the period condition, the symmetries, the expansion at the saddle point and the
boundary error, the dimension of the space of bounded Jacobi fields and the absence of an equivariant kernel. The second
kind computes the free-boundary correction near the hole by Newton's method. This gives the size of the correction, its
profile and the pictures. Figures~\ref{fig:intro}--\ref{fig:collar} are explained where they are first used. Figure~\ref{fig:intro}
uses $\varepsilon=0.45$, which is large enough to see the hole but outside the range where the theorem is proved.
Interactive three-dimensional models of the surfaces of Figures~\ref{fig:intro} and~\ref{fig:chm} are at \viewer{costa},
\viewer{chm2} and \viewer{chm3}, and the scripts are attached to the arXiv submission as ancillary files (Section~\ref{sec:num_scripts}).

\subsection*{Organization}
Section~\ref{sec:costa} collects the facts about the Costa surface: its Weierstrass data, the expansion at the saddle
point, the symmetry group $G$, the Fourier modes of $G$-equivariant functions and the absence of $G$-equivariant
bounded Jacobi fields. Section~\ref{sec:perturb} sets up the perturbations whose boundary stays on the sphere and
computes the initial error and the linearization. Section~\ref{sec:linear} proves the uniform invertibility of the
linearized problem, and Section~\ref{sec:nonlinear} solves the nonlinear problem and proves Theorem~\ref{thm:A}.
Section~\ref{sec:chm} proves Theorem~\ref{thm:B} for the Costa--Hoffman--Meeks surfaces. Section~\ref{sec:remarks} contains the
index bound, remarks and open questions, and Section~\ref{sec:numerics} the numerical experiments.

\section{The Costa surface}\label{sec:costa}

\subsection{Weierstrass data and position}
Let $T=\C/(\Z+i\Z)$ be the square torus and $\wp$ its Weierstrass function. Its half-period values are
$\wp(\frac12)=e_1>0$, $\wp(\frac i2)=-e_1$ and $\wp(\omega)=0$, where $\omega=\frac{1+i}2$, and
$\wp'^2=4\wp^3-4e_1^2\wp$. Put $\mathfrak a=2\sqrt{2\pi}\,e_1$ and
\[
\Phi=\Bigl(\tfrac12\bigl(\wp-\mathfrak a^2\tfrac{\wp}{\wp'^2}\bigr),\ \tfrac i2\bigl(\wp+\mathfrak a^2\tfrac{\wp}{\wp'^2}\bigr),\ \mathfrak a\tfrac{\wp}{\wp'}\Bigr)\,dz,
\qquad \x(z)=\Ree\int_\omega^z\Phi .
\]
These are the Weierstrass data $g=\mathfrak a/\wp'$, $dh=\mathfrak a\wp\,dz/\wp'$ in the convention
$\Phi=\bigl(\frac12(g^{-1}-g),\frac i2(g^{-1}+g),1\bigr)dh$, and $\x$ is the Costa surface
$M=\x(T\setminus\{0,\frac12,\frac i2\})$~\cite{Costa1984,HoffmanMeeks1985}: a complete embedded minimal surface of
genus one and total curvature $-12\pi$ with a planar end at $z=0$ and catenoidal ends at $z=\frac12,\frac i2$.
The value of $\mathfrak a$ is forced by the period condition. Indeed $\frac{\wp}{\wp'}dz=\frac1{8e_1}\,d\log\frac{\wp-e_1}{\wp+e_1}$,
so $dh$ has purely imaginary periods and real residues $\pm\mathfrak a/(4e_1)$ at the catenoidal ends; and
$\frac{\wp}{\wp'^2}=\frac{\wp(z+\frac12)-\wp(z+\frac i2)-2e_1}{16e_1^3}$ by the addition formula, so $g\,dh$ and $dh/g=\wp\,dz$
have no residues. The horizontal periods vanish if and only if $\oint g\,dh=\overline{\oint dh/g}$ on both generators of the
homology of $T$. With $\eta_1=\zeta(\frac12)=\frac\pi2$ and $\zeta(\frac i2)=-i\eta_1$ (Legendre's relation for the square lattice) one
finds $\oint_a\wp\,dz=-\pi$, $\oint_b\wp\,dz=i\pi$, $\oint_a\frac{\wp}{\wp'^2}dz=-\frac1{8e_1^2}$, $\oint_b\frac{\wp}{\wp'^2}dz=-\frac i{8e_1^2}$, so both
conditions say $\mathfrak a^2=8\pi e_1^2$. The unit normal is
$\n=(2\Ree g,2\Imm g,|g|^2-1)/(|g|^2+1)$.

Figure~\ref{fig:costa} shows the surface near its saddle point. Figure~\ref{fig:torus} shows the conformal picture: the
fundamental square of $T$, the three ends, the saddle point $\omega$, the curves fixed by the symmetries of
Section~\ref{sec:symmetries} and the level curves of $|\x|$. The disk $D_\varepsilon$ of Section~\ref{sec:perturb} is the region
inside one of these level curves.

\begin{figure}[ht]
\centering
\includegraphics[width=0.62\textwidth]{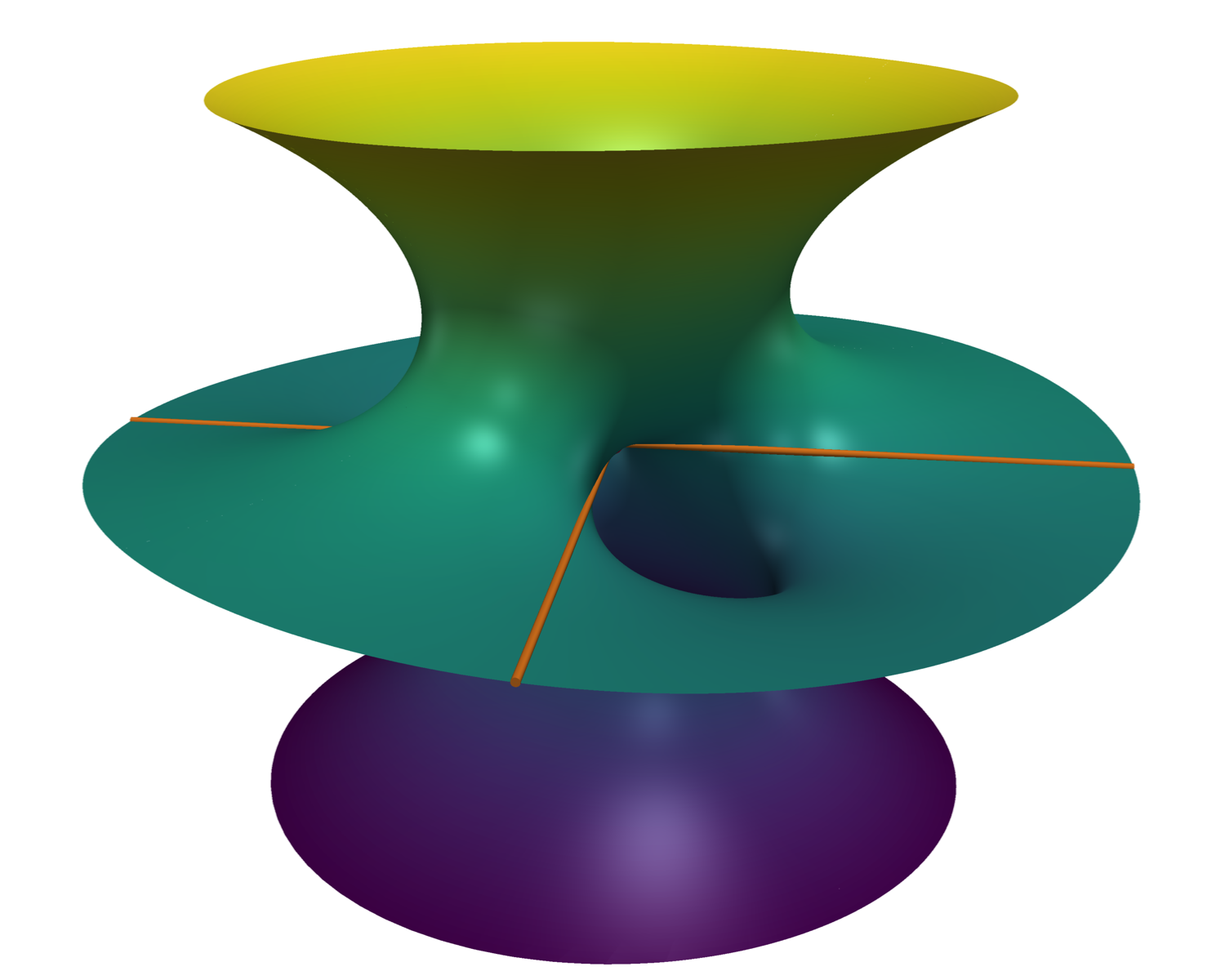}
\caption{The Costa surface $M$ in the position of Section~\ref{sec:costa}, cut off at $|x|=4$ and coloured by height.
The catenoidal ends open upwards and downwards, and the planar end is the horizontal sheet. The orange lines are the
straight lines $\{x_1=\pm x_2,\ x_3=0\}$ contained in $M$, the axes of the rotations $\tau$ and $\sigma_1\sigma_2\tau$; they cross at
the saddle point, the origin. Near the saddle point the planar sheet opens into the two catenoidal ends. Theorem~\ref{thm:A}
removes a small ball around this point.}
\label{fig:costa}
\end{figure}

\begin{figure}[ht]
\centering
\includegraphics[width=0.55\textwidth]{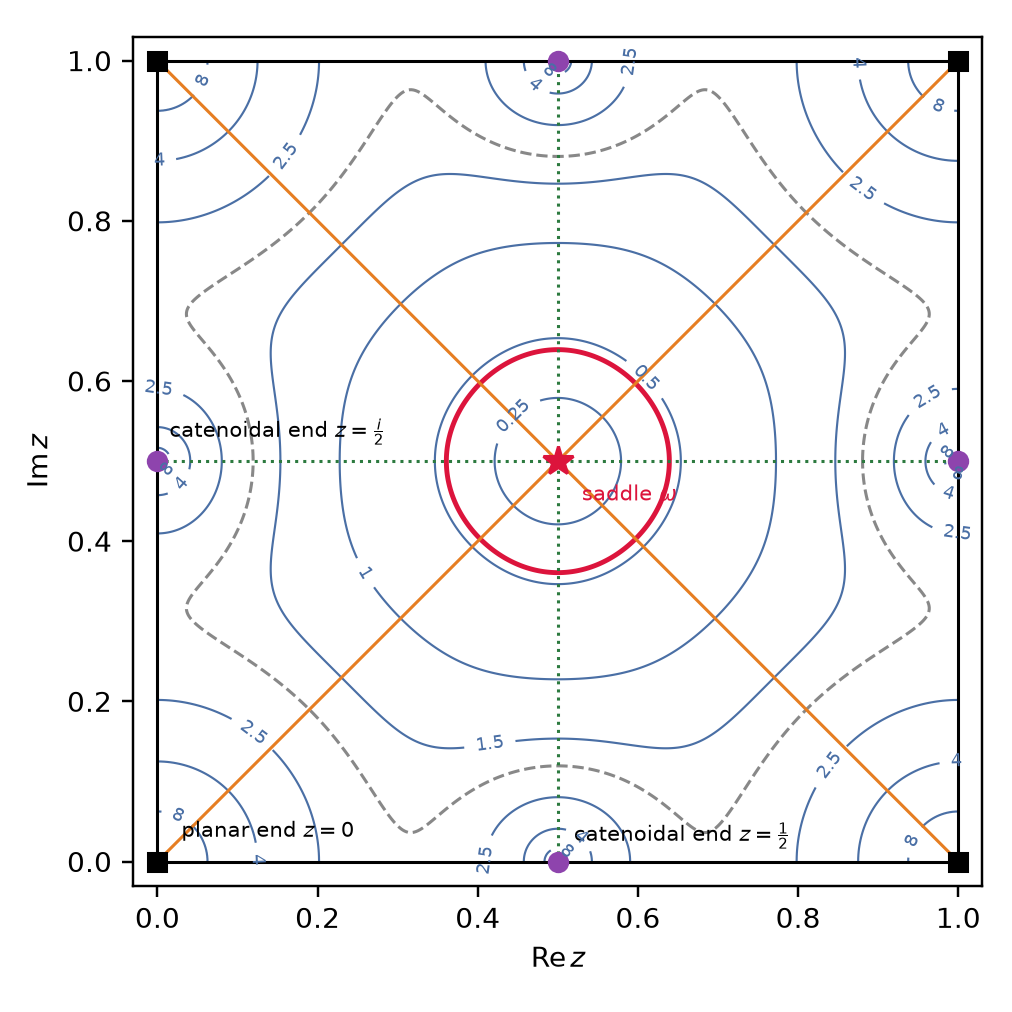}
\caption{The fundamental square of $T=\C/(\Z+i\Z)$ in the coordinate $z$, computed from the Weierstrass data.
Squares: the planar end $z=0$. Dots: the catenoidal ends $z=\frac12$ and $z=\frac i2$. Star: the saddle point
$\omega=\frac{1+i}2$. Blue: level curves of $|\x|$. Red: $|\x|=0.45$, the boundary of the disk $D_\varepsilon$ for the value of
$\varepsilon$ in Figure~\ref{fig:intro}. Dashed: $|\x|=1.883$; on the $1441$ sampled rays from $\omega$, every sublevel set $\{|\x|<\rho\}$ with
$\rho$ below this value is a disk around $\omega$. This is a numerical observation, not a proof (Section~\ref{sec:num_disk}). The diagonals (orange) are fixed by $\iota$ and by $\kappa_1\kappa_2\iota$ and are mapped
onto the two straight lines of Figure~\ref{fig:costa}. The dotted midlines are fixed by $\kappa_1$ and $\kappa_2$ and are mapped
into the symmetry planes $\{x_1=0\}$ and $\{x_2=0\}$.}
\label{fig:torus}
\end{figure}

\begin{lemma}[The saddle point]\label{lem:saddle}
$\x(\omega)=0$, $\n(\omega)=e_3$, and with $w=z-\omega$,
\[
\x(\omega+w)=\bigl(\pi\Ree w,\ \pi\Imm w,\ \tfrac{\mathfrak a}4\Ree w^2\bigr)+O(|w|^3).
\]
Consequently $M$ is, near the origin, the graph $x_3=\frac\kappa2(x_1^2-x_2^2)+O(|x'|^4)$ with
$\kappa=\mathfrak a/(2\pi^2)$, where $x'=(x_1,x_2)$.
\end{lemma}

\begin{proof}
$\wp(\omega+w)$ is even in $w$ with a double zero, since $\wp-e_2$ has a double zero at the half-period
$\omega$ and $e_2=0$; as $\wp''(\omega)=6e_2^2-2e_1^2=-2e_1^2$, we get $\wp=-e_1^2w^2+O(w^4)$ and
$\wp'=-2e_1^2w+O(w^3)$. Hence $\wp/\wp'^2=-1/(4e_1^2)+O(w^2)$ and, as $\mathfrak a^2=8\pi e_1^2$,
$\Phi=(\pi+O(w^2),\,-i\pi+O(w^2),\,\frac{\mathfrak a}2w+O(w^3))\,dw$. Integrating from $w=0$ gives the expansion,
and $g=\mathfrak a/\wp'$ has a pole at $\omega$, so $\n(\omega)=e_3$. The last sentence follows because
$x_1+ix_2=\pi w+O(|w|^3)$ and the third component is even in $w$ up to $O(|w|^4)$.
\end{proof}

\subsection{Symmetries}\label{sec:symmetries}
Let $\sigma_1,\sigma_2,\tau\in O(3)$ be
\[
\sigma_1(x)=(-x_1,x_2,x_3),\qquad \sigma_2(x)=(x_1,-x_2,x_3),\qquad \tau(x)=(x_2,x_1,-x_3),
\]
the reflections in the planes $\{x_1=0\}$, $\{x_2=0\}$ and the rotation by $\pi$ about the line
$\{x_1=x_2,\ x_3=0\}$. They generate a group $G$ of order $8$, which contains the rotoreflection
$S=\sigma_1\tau\colon x\mapsto(-x_2,x_1,-x_3)$. Put $\chi(\sigma)=\langle\sigma e_3,e_3\rangle\in\{\pm1\}$; this is a
homomorphism with $\chi(\sigma_1)=\chi(\sigma_2)=1$ and $\chi(\tau)=-1$.

\begin{lemma}\label{lem:symmetry}
The anti-holomorphic maps of $T$
\[
\kappa_1(z)=1-\bar z,\qquad \kappa_2(z)=\bar z+i,\qquad \iota(z)=i\bar z
\]
fix $\omega$ (note $i\bar\omega=\omega$) and permute the ends ($\iota$ exchanges $\frac12$ and $\frac i2$), and satisfy $\x\circ\kappa_1=\sigma_1\circ\x$, $\x\circ\kappa_2=\sigma_2\circ\x$,
$\x\circ\iota=\tau\circ\x$. Consequently, since $\x$ is injective, for every $\sigma\in G$ there is a unique isometry
$\kappa_\sigma$ of the flat torus with $\x\circ\kappa_\sigma=\sigma\circ\x$, and $\sigma\mapsto\kappa_\sigma$ is a homomorphism (only the
generators need be anti-holomorphic; e.g. $\kappa_S(z)=iz$). Moreover $\n\circ\kappa_\sigma=\chi(\sigma)\,\sigma\circ\n$. In particular $M$ is $G$-invariant, and $G$ fixes the origin.
\end{lemma}

\begin{proof}
The lattice is real and square, so $\wp(\bar z)=\overline{\wp(z)}$, $\wp(iz)=-\wp(z)$ and $\wp'(iz)=i\wp'(z)$.
For $\kappa_2$: $\wp(\kappa_2z)=\overline{\wp(z)}$, $\wp'(\kappa_2z)=\overline{\wp'(z)}$ and $d(\kappa_2z)=d\bar z$, so
$\kappa_2^*\Phi=(\overline{\Phi_1},-\overline{\Phi_2},\overline{\Phi_3})$, because $\mathfrak a$ is real. Integrating from the fixed
point $\omega$ gives $\x\circ\kappa_2=(x_1,-x_2,x_3)$. For $\kappa_1$: $\wp(1-\bar z)=\overline{\wp(z)}$,
$\wp'(1-\bar z)=-\overline{\wp'(z)}$ and $d(1-\bar z)=-d\bar z$, so $\kappa_1^*\Phi=(-\overline{\Phi_1},\overline{\Phi_2},\overline{\Phi_3})$.
For $\iota$: $\wp(i\bar z)=-\overline{\wp(z)}$, $\wp'(i\bar z)=i\,\overline{\wp'(z)}$ and $d(i\bar z)=i\,d\bar z$; hence
$\wp/\wp'^2$ is mapped to $\overline{\wp/\wp'^2}$ and $\iota^*\Phi=(\overline{\Phi_2},\overline{\Phi_1},-\overline{\Phi_3})$, which gives
$\x\circ\iota=(x_2,x_1,-x_3)$. Finally $\n\circ\kappa_\sigma$ and $\sigma\circ\n$ are both unit normals of $M$ at the
same points, so they agree up to a global sign on the connected surface; at $\omega$ they are $e_3$ and
$\sigma e_3=\chi(\sigma)e_3$.
\end{proof}

We call a function $u$ on (a $G$-invariant subset of) $M$, or of $T$, \emph{$G$-equivariant} if
\[
u(\sigma p)=\chi(\sigma)\,u(p)\qquad\text{for all }\sigma\in G,
\]
and denote such classes by a superscript $G$. By Lemma~\ref{lem:symmetry}, the normal graph of $u$ over $M$ is
$G$-invariant if $u$ is $G$-equivariant; the same holds for the graphs of Section~\ref{sec:perturb}.

\begin{lemma}[Fourier modes]\label{lem:modes}
Let $u$ be $G$-equivariant on a punctured neighbourhood of the saddle point, and write it in polar
coordinates $(r,\varphi)$ of the horizontal projection, or of the coordinate $w$. Then
$u=\sum_{n\equiv2\ (4)}u_n(r)\cos n\varphi$. The same holds on the planar end, in polar coordinates of its
horizontal projection, and in the conformal coordinate $z$ at $z=0$, where $\kappa_1,\kappa_2,\iota$ act by
$z\mapsto-\bar z,\bar z,i\bar z$; in particular a $G$-equivariant $C^2$ function on $\overline T$ vanishes at $z=0$ together with its
gradient. On the catenoidal end at $z=\frac12$ only the cosines of even frequencies occur, and $u$
on the other catenoidal end is determined by $u(\tau p)=-u(p)$.
\end{lemma}

\begin{proof}
Near the saddle, $\sigma_2$, $\sigma_1$ and $\tau$ act on the angle by $\varphi\mapsto-\varphi$, $\varphi\mapsto\pi-\varphi$ and
$\varphi\mapsto\frac\pi2-\varphi$, with $\chi=1,1,-1$. Invariance under the first gives cosines only, under the second
$(-1)^n=1$, and $\cos n(\frac\pi2-\varphi)=(-1)^{n/2}\cos n\varphi$ for even $n$, so the third forces $n/2$ odd. The planar
end is fixed by $G$ and is a graph over the exterior of a disk in $\{x_3=0\}$, where $G$ acts in the same way.
The catenoidal end at $\frac12$ is preserved by $\sigma_1,\sigma_2$ only, which gives even cosines, and $\tau$ exchanges the
two catenoidal ends.
\end{proof}

\subsection{Nondegeneracy}
By Fischer-Colbrie~\cite{FischerColbrie1985}, the bounded Jacobi fields of $M$ are the restrictions to $M$ of the null
space of a Schr\"odinger operator on the compactified torus (Section~\ref{sec:linear}), and Nayatani computed
its dimension.

\begin{theorem}[Nayatani~\cite{Nayatani1993}]\label{thm:nayatani}
The space of bounded solutions of the Jacobi equation $Lu=\Delta_Mu+|A|^2u=0$ on the Costa surface has
dimension $4$.
\end{theorem}

\begin{lemma}\label{lem:no_kernel}
No nonzero bounded Jacobi field of $M$ is $G$-equivariant.
\end{lemma}

\begin{proof}
The functions $\langle\n,a\rangle$, $a\in\R^3$, and $\rho=\langle\n,e_3\times\x\rangle$ are Jacobi fields, generated by
translations and by the rotation about the $x_3$-axis. They are bounded: this is clear for $\langle\n,a\rangle$, and on each
end, which is a graph $x_3=a_j\log r+b_j+O(r^{-1})$ with $\nabla x_3=a_jx'/r^2+O(r^{-2})$ (Schoen~\cite{Schoen1983}), the
horizontal part of $\n$ is $-a_jx'/r^2+O(r^{-2})$, which is orthogonal to $e_3\times\x=(-x_2,x_1,0)$ up to
$O(r^{-1})$. For $\sigma\in G$, by Lemma~\ref{lem:symmetry},
$\langle\n(\sigma p),a\rangle=\chi(\sigma)\langle\n(p),\sigma^{-1}a\rangle$, and, since
$\sigma a\times\sigma b=\det(\sigma)\,\sigma(a\times b)$ and $\sigma e_3=\chi(\sigma)e_3$,
$\rho(\sigma p)=\chi(\sigma)\langle\sigma\n,\chi(\sigma)\det(\sigma)\sigma(e_3\times p)\rangle=\det(\sigma)\rho(p)$. The four fields are linearly
independent. The $\langle\n,e_i\rangle$ are: the Gauss map extends to a map of degree $3$ from the compact torus $\overline T$ onto $S^2$
(on $M$ itself it omits $-e_3$, which it takes only at the planar end $z=0$), and $\langle\n,a\rangle\equiv0$ on $M$ extends to
$\overline T$ by continuity, which forces $a=0$. If
$\rho=\langle\n,a\rangle$, applying $\sigma_1$ and $\sigma_2$ (for which $\chi=1$, $\det=-1$) gives $\langle\n,\sigma_ia+a\rangle=0$, so
$\sigma_1a=\sigma_2a=-a$ and $a=0$; then $\rho\equiv0$, so $e_3\times\x$ is tangent to $M$ and $M$ is rotationally symmetric, which it is
not. By Theorem~\ref{thm:nayatani} the four fields span the bounded Jacobi fields.

Let $u=\langle\n,a\rangle+c\rho$ be $G$-equivariant. For $\sigma=\sigma_2$ equivariance gives $\langle\n,\sigma_2a-a\rangle-2c\rho=0$, so $c=0$ and $\sigma_2a=a$ by independence. Then, for
every $\sigma$, $\langle\n,\sigma^{-1}a-a\rangle=0$, so $\sigma^{-1}a=a$ for all $\sigma\in G$. The only vector fixed by $S$ is $0$.
\end{proof}

Section~\ref{sec:num_spectrum} confirms Theorem~\ref{thm:nayatani} and Lemma~\ref{lem:no_kernel} numerically (Figure~\ref{fig:spectrum}).

\section{Perturbations with boundary on the sphere}\label{sec:perturb}

By Lemma~\ref{lem:saddle} and since $\x$ is an embedding with $\x^{-1}(0)=\{\omega\}$, there is $\rho_0\in(0,\frac1{10})$ such that
$\x^{-1}(B_{4\rho_0})$ is a disk $U\ni\omega$ on which $|\x|$ has no critical point except $\omega$, and $M\cap B_{4\rho_0}$ is the graph
of Lemma~\ref{lem:saddle} with $|\nabla x_3|\leq\frac1{10}$. For $0<\varepsilon<\rho_0$ put
\[
D_\varepsilon=\{z\in U:\ |\x(z)|<\varepsilon\},\qquad \Omega_\varepsilon=\overline T\setminus D_\varepsilon ,
\]
where $\overline T=T$ is the compact torus, the three ends being ordinary points of it. $D_\varepsilon$ is a $G$-invariant
disk and $\x(\partial D_\varepsilon)=M\cap\partial B_\varepsilon$.

\emph{The vector field.} Choose $R_\infty$ so large that outside $B_{R_\infty}$ the three ends of $M$ are graphs over the exterior
of a disk in $\{x_3=0\}$, $x_3=f_j(x')$, with $|n_3|\geq\frac34$~\cite{Schoen1983}, and let
$S_j=\{x:\ |x'|>R_\infty,\ |x_3-f_j(x')|<\delta\}$ be vertical $\delta$-neighbourhoods of them; for small $\delta$ they are disjoint, since the
ends are disjoint graphs with $|f_i-f_j|\to\infty$, and $G$ permutes them, since it preserves vertical distances. Let $\varsigma=+1$ on
the slabs of the two catenoidal ends and $\varsigma=-1$ on that of the planar end; then $\n\to\varsigma\,e_3$ at each end, since $g=\infty$ at
$z=\frac12,\frac i2$ and $g=0$ at $z=0$, and $\operatorname{sign}n_3=\varsigma$ on $M\cap S_j$ because $|n_3|\geq\frac34$ there and each end is connected. Let $\eta,\eta_\infty$ be cut-off functions of $|x|$ with $\eta=1$ for $|x|\leq2\rho_0$, $\eta=0$
for $|x|\geq3\rho_0$, $\eta_\infty=0$ for $|x|\leq R_\infty$ and $\eta_\infty=1$ for $|x|\geq2R_\infty$, let $\widetilde{\n}$ be the $G$-equivariant
extension of $\n$ to a tubular neighbourhood of $M\cap(B_{3R_\infty}\setminus B_{\rho_0})$ by nearest-point projection, and put
\[
W(x)=\eta\Bigl(e_3-\frac{x_3}{|x|^2}\,x\Bigr)+\bigl(1-\eta-\eta_\infty\bigr)\widetilde{\n}(x)+\eta_\infty\,\varsigma\,e_3 ,
\]
defined on a neighbourhood of $M\setminus\{0\}$ (the supports of $\eta$ and $\eta_\infty$ are disjoint). For $0<|x|\leq2\rho_0$ the field $W$ is
the gradient of $x_3$ along the sphere $\partial B_{|x|}$; each such sphere is invariant under its flow $\Phi^t$, which moves points along
meridians, and flow lines starting outside $\overline B_{2\rho_0}$ never enter it. On $S_j\cap\{|x|\geq2R_\infty\}$ the flow is vertical translation.
$W$ satisfies $W(\sigma x)=\chi(\sigma)\sigma W(x)$ for $\sigma\in G$, since $\tau$ exchanges the catenoidal slabs and $\sigma e_3=\chi(\sigma)e_3$. The
function $\theta:=\langle W,\n\rangle$ satisfies $\theta\geq\frac12$ on $M$: near the origin $\theta=1+O(|x|^2)$ by Lemma~\ref{lem:saddle}
(after decreasing $\rho_0$), $\theta=1$ where $\eta=\eta_\infty=0$, and $\theta\geq\frac34$ on the ends. Throughout, $H$ denotes the trace of the second
fundamental form.

\emph{The surfaces.} For a function $u$ on $\Omega_\varepsilon$ with $|u|$ small (in the sense of Section~\ref{sec:nonlinear}), let
\[
\Sigma_\varepsilon(u)=\Bigl\{\Phi^{u(p)/\theta(p)}\bigl(\x(p)\bigr):\ p\in\Omega_\varepsilon\setminus\{\text{ends}\}\Bigr\}.
\]
Its first-order normal displacement is $u$. Its boundary $\Phi^{u/\theta}(\x(\partial D_\varepsilon))$ lies on $\partial B_\varepsilon$ exactly,
because $|\x|=\varepsilon<2\rho_0$ there and the flow preserves $|x|$. If $u$ is $G$-equivariant, $\Sigma_\varepsilon(u)$ is $G$-invariant, because
$\sigma\Phi^t=\Phi^{\chi(\sigma)t}\sigma$ and $\theta\circ\kappa_\sigma=\theta$. On the ends, for $|\x(p)|\geq2R_\infty+1$ (where the flow lines for $|t|\leq1$ stay in the vertical region), $\Sigma_\varepsilon(u)$ is the graph
of $f_j+\varsigma\,u/\theta$ over the same horizontal point, with $\theta=|n_3|=(1+|\nabla f_j|^2)^{-1/2}$.
We look for $u$ such that $\Sigma_\varepsilon(u)$ is minimal and meets $\partial B_\varepsilon$ orthogonally, that is,
\[
H\bigl(\Sigma_\varepsilon(u)\bigr)=0\ \text{ in }\Omega_\varepsilon,\qquad
\Theta_\varepsilon(u):=\Bigl\langle\n_u,\frac{x}{|x|}\Bigr\rangle=0\ \text{ on }\partial\Omega_\varepsilon ,
\]
where $\n_u$ is the unit normal of $\Sigma_\varepsilon(u)$ at the image of $p$.

\begin{lemma}[The initial error]\label{lem:error}
$H(\Sigma_\varepsilon(0))=0$, and on $\partial\Omega_\varepsilon$, in the polar angle $\varphi$ of the horizontal projection,
$\Theta_\varepsilon(0)=-\frac\kappa2\,\varepsilon\cos2\varphi+O(\varepsilon^3)$, together with its tangential derivatives of order $\le2$
measured at the scale $\varepsilon$.
\end{lemma}

\begin{proof}
$\Sigma_\varepsilon(0)=M\setminus B_\varepsilon$ is minimal. With $x_3=f(x')=\frac\kappa2(x_1^2-x_2^2)+O(|x'|^4)$ and
$\n=(-\nabla f,1)/\sqrt{1+|\nabla f|^2}$, one has
$\langle\n,x\rangle=(f-x'\cdot\nabla f)/\sqrt{1+|\nabla f|^2}=-\frac\kappa2(x_1^2-x_2^2)+O(|x'|^4)$, and $|x'|=\varepsilon+O(\varepsilon^3)$ on
$\partial B_\varepsilon$. The derivative bounds follow from the same expansion, which can be differentiated.
\end{proof}

Figure~\ref{fig:defect} and Table~\ref{tab:defect} show $\Theta_\varepsilon(0)$ computed from the Weierstrass data (Section~\ref{sec:num_defect}).

\begin{lemma}[Linearization]\label{lem:linearization}
$\frac{d}{dt}\big|_{t=0}H(\Sigma_\varepsilon(tu))=Lu$, where $L=\Delta_M+|A|^2$, and
$\frac{d}{dt}\big|_{t=0}\Theta_\varepsilon(tu)=\sqrt{1-\Theta_\varepsilon(0)^2}\,\partial_\nu u+\frac u\varepsilon+R_\varepsilon u$, where $\nu$ is the outward
unit conormal of $\Omega_\varepsilon$ along $\partial D_\varepsilon$ (it points towards the origin) and $|R_\varepsilon u|\leq C\varepsilon|u|$.
\end{lemma}

\begin{proof}
The variation field $V=(u/\theta)W$ has normal component $u$, and a tangential component, which does not change the
mean curvature to first order because $H(\Sigma_\varepsilon(0))=0$; this gives $Lu$. On the boundary $V$ is tangent to the
sphere, so $\delta|x|=0$ and $\delta(x/|x|)=V/\varepsilon$, while $\delta\n=-\nabla u-A(V^T)$, where $V^T$ is the tangential part
of $V$. Write $x=\varepsilon\Theta_\varepsilon(0)\n+x^T$ with $|x^T|=\varepsilon\sqrt{1-\Theta_\varepsilon(0)^2}$ and $x^T/|x^T|=-\nu$. Then
$\delta\Theta=-\langle\nabla u,x^T\rangle/\varepsilon-\langle A(V^T),x^T\rangle/\varepsilon+\langle\n,V\rangle/\varepsilon$, which is the claim with
$R_\varepsilon u=-\langle A(V^T),x^T\rangle/\varepsilon$. Since $W^T=W-\theta\n=O(\varepsilon)$ on $\partial B_\varepsilon$ by Lemma~\ref{lem:saddle} and $|A|$ is
bounded near the origin, $|R_\varepsilon u|\leq C\varepsilon|u|$.
\end{proof}

\section{Linear theory}\label{sec:linear}

\subsection{The compactified Jacobi operator}
Let $\tilde g=|dz|^2$ be the flat metric of the torus $\overline T$. It is $G$-invariant, because all $\kappa_\sigma$ are isometries of the
square torus, and it equals $|dw|^2$ on $U$. Write the metric of $M$ as
$g_M=e^{2\phi}\tilde g$. Then $L=e^{-2\phi}J$ with
\[
J=\Delta_{\tilde g}+\tilde V,\qquad \tilde V=e^{2\phi}|A|^2 .
\]
In a local conformal coordinate $\zeta$ with $\tilde g=\mu^2|d\zeta|^2$ one has $\tilde V=8|g_\zeta|^2/\bigl(\mu^2(1+|g|^2)^2\bigr)$,
which is smooth on all of $\overline T$ because $g$ is meromorphic, also at the ends~\cite{FischerColbrie1985}. $J$
commutes with the action $u\mapsto\chi(\sigma)\,u\circ\kappa_\sigma^{-1}$.

\begin{proposition}\label{prop:J_iso}
$J\colon C^{2,\alpha}(\overline T)^G\to C^{0,\alpha}(\overline T)^G$ is an isomorphism.
\end{proposition}

\begin{proof}
$J$ is a formally self-adjoint elliptic operator on $L^2(\overline T,dA_{\tilde g})$, and the action
$u\mapsto\chi(\sigma)\,u\circ\kappa_\sigma^{-1}$ is unitary and commutes with $J$, so it commutes with the orthogonal projection $P$ onto the
$G$-equivariant functions. By Fredholm theory $f\in C^{0,\alpha}(\overline T)$ is in the range of $J$ if and only if $f\perp\ker J$; if $f$ is
$G$-equivariant, this holds if and only if $f\perp(\ker J)^G$, and then $Pu$ is a $G$-equivariant preimage of $f$ whenever $u$ is a
preimage. So on the $G$-equivariant functions the kernel is $(\ker J)^G$ and the range is the orthogonal complement of $(\ker J)^G$. An
element of $\ker J$ is continuous on $\overline T$, so it restricts to a bounded Jacobi field of $M$; by Lemma~\ref{lem:no_kernel},
$(\ker J)^G=0$.
\end{proof}

\subsection{Weighted norms on the domain with a hole}
Let $U_2=\x^{-1}(B_{2\rho_0})\subset U$. Let $\vartheta$ be a smooth $G$-invariant function on $\overline T\setminus\{\omega\}$ with
$\vartheta=|\x|$ on $U_2$, $2\rho_0\leq\vartheta\leq3\rho_0$ on $U\setminus U_2$ and $\vartheta=3\rho_0$ outside $U$; for instance
$\vartheta=\psi(|\x|)$ on $U$, with $\psi$ smooth and nondecreasing, $\psi(t)=t$ for $t\leq2\rho_0$ and $\psi(t)=3\rho_0$ for
$t\geq\frac72\rho_0$. On $\Omega_\varepsilon$ we use the scale-invariant norms
\[
\begin{aligned}
\|u\|_{X}&=\sup_{\Omega_\varepsilon}\Bigl(\vartheta^{-1}|u|+|\nabla u|+\vartheta|\nabla^2u|\Bigr)+\sup_{p}\vartheta(p)^{1+\alpha}[\nabla^2u]_{\alpha,B_p},\\
\|f\|_{Y}&=\sup_{\Omega_\varepsilon}\vartheta|f|+\sup_p\vartheta(p)^{1+\alpha}[f]_{\alpha,B_p},
\end{aligned}
\]
where $B_p=B(p,\vartheta(p)/(8\pi))\cap\Omega_\varepsilon$ and all quantities are computed with the flat metric $\tilde g$. Since
$\vartheta=|\x|=\pi|w|(1+O(|w|^2))$ on $U_2$, for $p\in U_2$ the ball $B_p$ has $\tilde g$-radius about $|w(p)|/8$ and stays in
$\{\vartheta\geq\frac12\vartheta(p)\}$. Outside $U_2$ the weight lies between $2\rho_0$ and $3\rho_0$, so there the weighted norms are
equivalent to the usual Hölder norms with constants independent of $\varepsilon$. On
$\partial\Omega_\varepsilon$, with $s$ the $\tilde g$-arclength,
\[
\|h\|_{Z}=\sup|h|+\varepsilon\sup|\partial_sh|+\varepsilon^{1+\alpha}[\partial_sh]_\alpha .
\]
The weight $\vartheta^{1}$ in $X$ is the natural one: it measures slopes. Note that $\|u\|_X<\infty$ allows $u$ to be of
size $\vartheta$ near the hole, and on $\Omega_\varepsilon$ these norms are equivalent to the usual Hölder norms, with constants depending
on $\varepsilon$. We write the boundary operator of Lemma~\ref{lem:linearization} as
\[
\mathcal B_\varepsilon u=\partial^M_\nu u+\frac u\varepsilon ,
\]
with $\partial^M_\nu=e^{-\phi}\partial^{\tilde g}_\nu$ the conormal derivative for $g_M$. For $u\in X$ one has $\|\mathcal B_\varepsilon u\|_Z\leq C\|u\|_X$.

\subsection{Three limiting problems}
In the following lemma $G$ acts on $\R^2$ by $y\mapsto\bar y,-\bar y,i\bar y$ with $\chi=1,1,-1$, as it does on the
coordinate $w$, and $G$-equivariant harmonic functions have only the Fourier modes $n\equiv2\pmod4$ by the
proof of Lemma~\ref{lem:modes}.

\begin{lemma}\label{lem:limits}
\begin{enumerate}[label=(\alph*)]
\item If $u$ is a $G$-equivariant solution of $Ju=0$ on $\overline T\setminus\{\omega\}$ with $|u|\leq C\vartheta$, then $u=0$.
\item If $u$ is $G$-equivariant and harmonic on $\{|y|\geq1\}\subset\R^2$, satisfies $-\partial_ru+u=0$ on $|y|=1$ and
$|u(y)|\leq C|y|$, then $u=0$.
\item If $u$ is $G$-equivariant and harmonic on $\R^2\setminus\{0\}$ and $|u(y)|\leq C|y|$, then $u=0$.
\end{enumerate}
\end{lemma}

\begin{proof}
(a) $u$ is bounded near $\omega$, so the singularity is removable, and $u\in\ker J^G=0$ by Proposition~\ref{prop:J_iso}.
(b), (c) By the proof of Lemma~\ref{lem:modes}, all sine coefficients of $u(r,\cdot)$ and all cosine coefficients with
$n\not\equiv2\pmod4$ vanish. For $n\equiv2\pmod4$ the coefficient $u_n(r)=\frac1\pi\int_0^{2\pi}u(r,\varphi)\cos n\varphi\,d\varphi$ is $C^2$ in $r$
(and $C^1$ up to $r=1$ in (b)) and satisfies $u_n''+u_n'/r-n^2u_n/r^2=0$, so $u_n=a_nr^n+b_nr^{-n}$. The bound $|u_n(r)|\leq2Cr$ forces $a_n=0$ as $r\to\infty$,
because $n\geq2>1$. In (c) it also forces $b_n=0$ as $r\to0$. In (b) the Robin condition gives
$nb_n+b_n=0$, so $b_n=0$.
\end{proof}

\subsection{Uniform invertibility}

\begin{proposition}\label{prop:uniform}
There are $\varepsilon_1>0$ and $C$ such that for $\varepsilon\in(0,\varepsilon_1)$ and every $G$-equivariant $u\in C^{2,\alpha}(\Omega_\varepsilon)$,
\[
\|u\|_X\leq C\bigl(\|Ju\|_Y+\|\mathcal B_\varepsilon u\|_Z\bigr).
\]
Consequently $(J,\mathcal B_\varepsilon)\colon X^G\to Y^G\times Z^G$ is an isomorphism whose inverse has norm at most $C$.
\end{proposition}

\begin{proof}
\emph{Step 1: reduction to the weighted supremum.} Interior Schauder estimates on the balls $B(p,\vartheta(p)/(4\pi))$ with
$\vartheta(p)\geq2\varepsilon$ (they do not meet $D_\varepsilon$), rescaled to unit size, and the Schauder estimates for the oblique derivative
problem~\cite[Section~6.7]{GilbargTrudinger} on $\{\varepsilon\leq\vartheta\leq4\varepsilon\}$, rescaled by $\varepsilon$, give
$\|u\|_X\leq C_0\bigl(\sup\vartheta^{-1}|u|+\|Ju\|_Y+\|\mathcal B_\varepsilon u\|_Z\bigr)$ with $C_0$ independent of $\varepsilon$. The rescaled
problems have uniformly bounded coefficients: after rescaling by $\vartheta(p)$ near the hole, $\tilde g$ is flat, the
rescaled potential is $\vartheta^2\tilde V=O(\vartheta^2)$, and in the variable $y=\pi w/\varepsilon$ the boundary $\{|\x|=\varepsilon\}$ is the curve $|y|^2=1+O(\varepsilon^2)$, which converges
to the unit circle in $C^\infty$ by Lemma~\ref{lem:saddle}. With $u=\varepsilon\tilde u(y)$ one has
$\mathcal B_\varepsilon u=\pi e^{-\phi}\partial^y_\nu\tilde u+\tilde u$, and $\pi e^{-\phi}=1+O(\varepsilon^2)$ because $e^\phi=|\partial\x/\partial(\Ree w)|=\pi+O(|w|^2)$ on $U$.

\emph{Step 2: contradiction.} If the estimate fails, there are $\varepsilon_i\to0$ and $G$-equivariant $u_i$ with
$\|u_i\|_X=1$ and $\|Ju_i\|_Y+\|\mathcal B_{\varepsilon_i}u_i\|_Z\to0$. By Step~1, $\sup\vartheta^{-1}|u_i|\geq c>0$; choose $p_i$ with
$\vartheta(p_i)^{-1}|u_i(p_i)|\geq c/2$. Pass to a subsequence.

(a) If $\vartheta(p_i)\geq c'>0$, the bound $\|u_i\|_X\le1$ gives $C^{2,\alpha}$ bounds on compact subsets of $\overline T\setminus\{\omega\}$, so
$u_i\to u$ in $C^2_{\rm loc}$ with $Ju=0$, $|u|\leq\vartheta$, $u$ $G$-equivariant and $|u(p_\infty)|\geq cc'/2$. This contradicts
Lemma~\ref{lem:limits}(a).

(b) If $\vartheta(p_i)/\varepsilon_i\leq C'$, put $y=\pi w/\varepsilon_i$ and $\tilde u_i(y)=\varepsilon_i^{-1}u_i$. The rescaled domains converge to
$\{|y|\geq1\}$, the rescaled operators to $\Delta$ and $-\partial_r+1$, with right-hand sides tending to zero, and
$|\tilde u_i(y)|\leq|y|(1+o(1))$. A subsequence converges in $C^2_{\rm loc}$ up to the boundary to a solution of the problem of
Lemma~\ref{lem:limits}(b) which does not vanish at the limit of $\pi w(p_i)/\varepsilon_i$.

(c) If $\vartheta(p_i)\to0$ and $\vartheta(p_i)/\varepsilon_i\to\infty$, put $y=\pi w/\vartheta(p_i)$ and $\tilde u_i=\vartheta(p_i)^{-1}u_i$. The
holes shrink to the origin and the limit is a solution of the problem of Lemma~\ref{lem:limits}(c), nonzero at a point
with $|y|=1$.

All three cases are impossible, which proves the estimate.

\emph{Step 3: invertibility.} For fixed $\varepsilon$, $\Omega_\varepsilon$ is a compact surface with smooth boundary and $(J,\mathcal B_\varepsilon)$ is a
Robin problem with smooth coefficients. It is symmetric: by Green's formula for $\tilde g$ and
$\partial^{\tilde g}_\nu u=-e^{\phi}u/\varepsilon$ on the boundary,
$\int(Ju)v-u(Jv)\,dA_{\tilde g}=\int_{\partial\Omega_\varepsilon}\bigl(\partial^{\tilde g}_\nu u\,v-u\,\partial^{\tilde g}_\nu v\bigr)ds=0$ when $\mathcal B_\varepsilon u=\mathcal B_\varepsilon v=0$.
The associated form $\int_{\Omega_\varepsilon}\bigl(|\nabla u|^2-\tilde Vu^2\bigr)dA_{\tilde g}+\int_{\partial\Omega_\varepsilon}\varepsilon^{-1}e^{\phi}u^2\,ds$ is
bounded below on $H^1$, its boundary term being positive, so $-J+\lambda$ with the Robin condition is invertible for large
$\lambda$ by Lax--Milgram and elliptic regularity for the oblique derivative problem (\cite[Section~6.7]{GilbargTrudinger}, stated there for
domains in $\R^n$ and extended to $\Omega_\varepsilon$ by a partition of unity), and $(J,\mathcal B_\varepsilon)$ is a compact perturbation of an isomorphism. Hence it
is Fredholm of index zero between Hölder spaces, also on the $\chi$-isotypic subspace, which it preserves, and by the
estimate it is injective for $\varepsilon<\varepsilon_1$. It is therefore an isomorphism, and the estimate bounds the inverse.
\end{proof}

\section{The nonlinear problem}\label{sec:nonlinear}

Set
\[
\mathcal N_\varepsilon(u)=\Bigl(e^{2\phi}H\bigl(\Sigma_\varepsilon(u)\bigr),\ \Theta_\varepsilon(u)\Bigr),
\]
defined for $G$-equivariant $u\in X$ with $\|u\|_X\leq c_0$. By Lemmas~\ref{lem:error} and~\ref{lem:linearization},
\begin{equation}\label{eq:expansion}
\mathcal N_\varepsilon(u)=\bigl(0,\Theta_\varepsilon(0)\bigr)+\bigl(Ju,\mathcal B_\varepsilon u\bigr)+\bigl(0,\mathcal R_\varepsilon u\bigr)+\mathcal Q_\varepsilon(u),
\end{equation}
with $\mathcal R_\varepsilon u=(\sqrt{1-\Theta_\varepsilon(0)^2}-1)\partial_\nu u+R_\varepsilon u$ and $\mathcal Q_\varepsilon$ the remainder, which vanishes to second
order at $u=0$.

The ends are handled in the inversion coordinate, by the following explicit formulas.

\begin{lemma}[The ends in the inversion coordinate]\label{lem:ends}
Let $f(x')=a\log|x'|+b+h(y)$ on $\{|x'|>R\}\subset\R^2$, where $y=x'/|x'|^2$ and $h$ is smooth near $y=0$. With $\nabla h$ and
$D^2h$ taken in $y$, put
\[
s=|y|^2,\quad P=sI-2yy^T,\quad c=y\cdot\nabla h,\quad d=|\nabla h|^2,\quad B=a^2-2ac+sd,\quad v=-ay+s\nabla h .
\]
Then the following hold.
\begin{enumerate}[label=(\roman*)]
\item $\nabla_{x'}f=ay+P\nabla h$, and $|\nabla_{x'}f|^2=sB=|v|^2$.
\item In the coordinate $y$ the metric of the graph of $f$ is $|y|^{-4}(I+vv^T)$.
\item The minimal surface operator $\mathcal M(f)=\operatorname{div}\bigl(\nabla f/\sqrt{1+|\nabla f|^2}\bigr)$ satisfies
\[
|y|^{-4}\mathcal M(f)=\frac{(1+sB)\Delta h-v^TD^2h\,v+a^3-4a^2c+2ac^2+3asd-2csd}{(1+sB)^{3/2}} .
\]
\end{enumerate}
The right-hand side of (iii) is a smooth function of $(y,\nabla h,D^2h)$, also at $y=0$, affine in $D^2h$, and its principal part
$(1+sB)I-vv^T$ has the eigenvalues $1$ and $1+sB$.
\end{lemma}

\begin{proof}
Since $x'=y/|y|^2$, one has $D_{x'}y=P$, $P^2=s^2I$ and $Py=-sy$. Hence $\nabla_{x'}\log|x'|=x'/|x'|^2=y$ and
$\nabla_{x'}(h\circ y)=P\nabla h$, which gives the first formula in (i); then $|ay+P\nabla h|^2=a^2s+2a\,y^TP\nabla h+s^2d=s(a^2-2ac+sd)$,
and $|v|^2$ is the same. For (ii), $D_yx'=P^{-1}=s^{-2}P$, so the metric $I+\nabla f\,\nabla f^T$ becomes
$s^{-4}P(I+\nabla f\,\nabla f^T)P=s^{-2}I+s^{-4}(P\nabla f)(P\nabla f)^T$, and $P\nabla f=-asy+s^2\nabla h=sv$. For (iii), write
$\operatorname{div}_{x'}F=\operatorname{tr}(P\,D_yF)$ for a vector field $F$ and expand $\operatorname{div}_{x'}(\nabla f/\sqrt{1+sB})$ by (i). The
result is an identity between rational functions of $y$, $a$ and the derivatives of $h$ of order at most two. It is a
direct computation, and we also verified it symbolically (Section~\ref{sec:num_scripts}). The last sentence follows from
$|v|^2=sB$.
\end{proof}

\begin{lemma}\label{lem:nonlinear}
There are $c_0>0$ and $C$, independent of $\varepsilon$, such that for $G$-equivariant $u,v$ with $\|u\|_X,\|v\|_X\leq c_0$,
\[
\|\Theta_\varepsilon(0)\|_Z\leq C\varepsilon,\qquad \|\mathcal R_\varepsilon u\|_Z\leq C\varepsilon\|u\|_X,\]
\[
\|\mathcal Q_\varepsilon(u)-\mathcal Q_\varepsilon(v)\|_{Y\times Z}\leq C\bigl(\|u\|_X+\|v\|_X\bigr)\|u-v\|_X .
\]
\end{lemma}

\begin{proof}
The first bound is Lemma~\ref{lem:error}. For the second, $\sqrt{1-\Theta_\varepsilon(0)^2}-1=O(\varepsilon^2)$ and $|\partial_\nu u|\leq\|u\|_X$;
and $|R_\varepsilon u|\leq C\varepsilon|u|\leq C\varepsilon^2\|u\|_X$ on the boundary, where $\vartheta=\varepsilon$; the tangential derivatives are estimated in the
same way.

For the third we use scaling. \emph{Near the hole} ($\vartheta\leq\rho_0$), fix $p$ and rescale $\R^3$ by $\vartheta(p)^{-1}$. The
rescaled piece of $M$ over $B(p,\vartheta(p)/(4\pi))$ (or over $\{\varepsilon\leq\vartheta\leq4\varepsilon\}$ when $\vartheta(p)\leq2\varepsilon$) is the graph of a function with $C^{3,\alpha}$ norm $O(\vartheta(p))$, the rescaled field
$W$ is homogeneous of degree zero near the origin and hence has bounded derivatives, and the rescaled function
$\hat u=\vartheta(p)^{-1}u$ has $C^{2,\alpha}$ norm at most $C\|u\|_X$. The rescaled mean curvature and boundary angle are smooth
functions of $(\hat u,\nabla\hat u,\nabla^2\hat u)$ and of the rescaled geometry, with uniformly bounded derivatives, so their
second-order remainders satisfy the quadratic estimate in $C^{0,\alpha}$ and $C^{1,\alpha}$ at unit scale. Scaling back, the mean
curvature scales like $\vartheta^{-1}$ and $e^{2\phi}=\pi^2+O(\vartheta^2)$ on $U$, which is exactly the weight $\vartheta$ in $\|\cdot\|_Y$; the
boundary angle is scale invariant, which matches $\|\cdot\|_Z$.
\emph{On the compact part} of $\overline T$ away from $\omega$ and from the ends, $\Sigma_\varepsilon(u)$ is the image of a fixed piece of $M$
under $\Phi^{u/\theta}$ with $W$ smooth and $\theta\geq\frac12$, and the estimate is the smooth dependence of the mean curvature on
$(u,\nabla u,\nabla^2u)$.
\emph{Near an end} ($|\x(p)|\geq2R_\infty+1$) the end is the graph $x_3=f(x')$ over $\{|x'|>R\}$ with
$f=a\log|x'|+b+h(y)$ as in Lemma~\ref{lem:ends}, and $a=0$ for the planar end. Here $h$ is smooth in $y$: the horizontal
coordinates of $M$ have no residues at the ends, so $y=x'/|x'|^2$ is a smooth (not conformal) coordinate of $\overline T$ at the end point,
and $x_3-a\log|x'|$ is a smooth function of it (see also~\cite{Schoen1983}). By Lemma~\ref{lem:ends}(ii), $g_M=|y|^{-4}(I+vv^T)$ in this
coordinate; since also $g_M=e^{2\phi}\tilde g$ with $\tilde g$ smooth, $e^{2\phi}|y|^4$ is a smooth positive function of $y$. Since $W=\varsigma\,e_3$
there, $\Sigma_\varepsilon(u)$ is the graph of $f+w$ over the same point, with $w=\varsigma\,u\sqrt{1+|\nabla f|^2}=\varsigma\,u\sqrt{1+sB}$;
multiplication by this smooth positive function of $y$ is an isomorphism of $C^{2,\alpha}$ in $y$. Lemma~\ref{lem:ends}, applied to
$h+w$ in place of $h$ (the coefficient $a$ does not change), gives $|y|^{-4}\mathcal M(f+w)=\mathcal A\bigl(y,\nabla(h+w),D^2(h+w)\bigr)$, with $\mathcal A$
smooth in all its arguments up to $y=0$ and affine in the last one. Hence $w\mapsto|y|^{-4}\mathcal M(f+w)$ is a smooth map from
$C^{2,\alpha}$ to $C^{0,\alpha}$ in $y$ near the end point, and its remainder after the linear term satisfies the quadratic estimate.
Moreover, every term of degree at least two in $w$ carries a factor $|y|^2$. Indeed, in Lemma~\ref{lem:ends}(iii) the first
derivatives of $h$ enter only through $sd$, through $c=y\cdot\nabla h$ and through $v=-ay+s\nabla h$; a product of two of these that both
depend on $w$ contains $s$, $|y|^2$ or $s|y|$, and the denominator depends on $w$ only through $sB=a^2s-2asc+s^2d$. With respect to $\n$ the mean curvature of the graph of $F$ is $\varsigma\,\mathcal M(F)$, $\mathcal M$ being computed for the upward
normal. So $e^{2\phi}H=\varsigma\,(e^{2\phi}|y|^4)\,|y|^{-4}\mathcal M(f+w)$, and this gives the quadratic estimate in $C^{0,\alpha}(\overline T)$ near
the end point, with a bound $O(|y|^2)$.
The different regions overlap on sets of bounded geometry, where the
estimates agree.
\end{proof}

\begin{proof}[Proof of Theorem~\ref{thm:A}]
Let $\mathcal G_\varepsilon=(J,\mathcal B_\varepsilon)^{-1}$ from Proposition~\ref{prop:uniform}. By~\eqref{eq:expansion}, $\mathcal N_\varepsilon(u)=0$ is equivalent to
the fixed-point equation
\[
u=\mathcal T_\varepsilon(u):=-\mathcal G_\varepsilon\Bigl(\bigl(0,\Theta_\varepsilon(0)+\mathcal R_\varepsilon u\bigr)+\mathcal Q_\varepsilon(u)\Bigr)
\]
in $X^G$. All operators preserve $G$-equivariance. Take $\varepsilon$ so small that $2C^2\varepsilon\leq c_0$. By Lemma~\ref{lem:nonlinear}, on the ball
$\{\|u\|_X\leq 2C^2\varepsilon\}$ one has $\|\mathcal T_\varepsilon(u)\|_X\leq C(C\varepsilon+C\varepsilon\|u\|_X+C\|u\|_X^2)\leq2C^2\varepsilon$ and
$\|\mathcal T_\varepsilon(u)-\mathcal T_\varepsilon(v)\|_X\leq C^2(\varepsilon+4C^2\varepsilon)\|u-v\|_X\leq\frac12\|u-v\|_X$ for $\varepsilon$ small. The contraction
mapping theorem gives a $G$-equivariant $u_\varepsilon$ with $\mathcal N_\varepsilon(u_\varepsilon)=0$ and $\|u_\varepsilon\|_X\leq C\varepsilon$. By elliptic regularity for the
nonlinear oblique boundary condition (bootstrapping the estimates of~\cite[Section~6.7]{GilbargTrudinger} by difference quotients),
$u_\varepsilon$ is smooth up to the boundary. Put $\Sigma_\varepsilon=\varepsilon^{-1}\Sigma_\varepsilon(u_\varepsilon)$.

\emph{Free boundary.} $\Sigma_\varepsilon(u_\varepsilon)$ is minimal, its boundary lies on $\partial B_\varepsilon$ (Section~\ref{sec:perturb}) and it meets
$\partial B_\varepsilon$ orthogonally, since $\Theta_\varepsilon(u_\varepsilon)=0$.

\emph{Exteriority.} For $\varepsilon\leq|\x(p)|\leq2\rho_0$ the flow preserves $|x|$, so the image point has norm $|\x(p)|$, which equals
$\varepsilon$ only on $\partial\Omega_\varepsilon$. For $|\x(p)|\geq2\rho_0$ the displacement is at most $C|u_\varepsilon(p)|\leq C\varepsilon$, so the image stays
outside $B_{\rho_0}$. Hence $\Sigma_\varepsilon(u_\varepsilon)\cap\overline B_\varepsilon=\partial\Sigma_\varepsilon(u_\varepsilon)$.

\emph{Embeddedness and ends.} We check injectivity in three regions of $M$: (1) the core $\{\varepsilon\leq|\x(p)|<\rho_0\}$, (2) the ends
$\{|\x(p)|\geq2R_\infty+1\}$, and (3) the compact part $\{\rho_0\leq|\x(p)|\leq3R_\infty\}$; (2) and (3) overlap. Images of the core lie in $B_{\rho_0}$
and all other images lie outside it, since the flow preserves $|x|$ for $|x|\leq2\rho_0$ and its flow lines do not enter $\overline B_{2\rho_0}$ from
outside. If two points $p,q$ outside the core have the same image, then $|\x(p)-\x(q)|\leq4C\varepsilon$, so for $\varepsilon$ small both lie in (2) or
both in (3). It therefore suffices to check injectivity within each region. (1) On $\{\varepsilon\leq|x|\leq2\rho_0\}$ the flow moves points along meridians of the
spheres $\partial B_r$ and keeps $r$ and the azimuth. By Lemma~\ref{lem:saddle}, for $\varepsilon\leq r\leq2\rho_0$ the curve $M\cap\partial B_r$ is a
graph over the azimuth in the band $|x_3|\leq r/5$, and moving each of its points along its meridian by at most
$|u_\varepsilon/\theta|\leq2C\varepsilon r$ gives again such a graph. Distinct spheres are disjoint, so this part of $\Sigma_\varepsilon(u_\varepsilon)$ is embedded.
(2) There $\Sigma_\varepsilon(u_\varepsilon)$ consists of the three vertical graphs of $f_j+v_j$ with $|v_j|\leq C\varepsilon$ over the exterior of a disk, and
$|f_i-f_j|\geq\delta>0$ on the common domain, so the perturbed graphs are disjoint for $C\varepsilon<\delta/2$.
(3) On this compact piece of the embedded surface $M$, where $W$ is smooth and transverse, $(p,t)\mapsto\Phi^t(\x(p))$ is an embedding
of $\{|t|<t_0\}$ for some $t_0>0$, and $|u_\varepsilon/\theta|\leq C\varepsilon<t_0$. In the coordinate $y$ at an end, $u_\varepsilon$ is $C^{2,\alpha}$ up to the end point. At the catenoidal ends
$\Sigma_\varepsilon(u_\varepsilon)$ is the graph of $f_j+v_j$ with $v_j\to u_\varepsilon(\text{end})$, a vertical translation, and the two translations are opposite
by $\tau$; at the planar end $u_\varepsilon$ vanishes to second order by Lemma~\ref{lem:modes}. So the three ends remain embedded, parallel,
of multiplicity one, of finite total curvature and of the same types.

\emph{Topology and the rest.} $\Sigma_\varepsilon$ is diffeomorphic to $\overline T\setminus\bigl(D_\varepsilon\cup\{\text{three points}\}\bigr)$: genus one,
one boundary circle, three ends; it is connected. It is proper, because $M$ is proper and the displacement is bounded.
It is $G$-invariant because $u_\varepsilon$ is $G$-equivariant. The bound
$\|u_\varepsilon\|_X\leq C\varepsilon$ gives (iii). For (iv), $\Sigma_\varepsilon$ is connected, orientable and properly immersed with
$\Sigma_\varepsilon\cap S^2=\partial\Sigma_\varepsilon$, and its ends are parallel, embedded and regular in the sense
of~\cite[(2.5)]{MazetMendes2022}: each is the graph of $f_j+v_j$ with $v_j$ of class $C^{2,\alpha}$ in $y=x'/|x'|^2$, so
$v_j=v_j(0)+\nabla v_j(0)\cdot y+O(|y|^2)$, which is the required expansion $a\log|x'|+b+(c,x')|x'|^{-2}+O(|x'|^{-2})$. It is not
catenoidal, that is, not invariant under the rotations about a line, because it has three ends. So it is unstable
by~\cite[Theorem~1.1]{MazetMendes2022}.
\end{proof}

\section{Costa--Hoffman--Meeks surfaces of every genus}\label{sec:chm}

The proof of Theorem~\ref{thm:A} uses the Costa surface through three properties: a point with horizontal tangent plane
at which the symmetries allow only Fourier modes of frequency at least two, nondegeneracy, and the absence of
equivariant bounded Jacobi fields. The Costa--Hoffman--Meeks surfaces $M_k$ of genus $k\geq1$~\cite{HoffmanMeeks1990} have all
three. This section proves Theorem~\ref{thm:B}. Throughout, $m=k+1$.

\subsection{The surfaces}\label{sec:chm_data}
Let $\overline M_k$ be the compact Riemann surface of the curve
\[
w^m=z^{m-1}(z^2-1),
\]
that is, its normalization: for $k>1$ the curve is singular at $(0,0)$, but the corresponding point of $\overline M_k$ is smooth. The
projection $(z,w)\mapsto z$ has degree $m$ and is totally ramified exactly over $0$, $1$, $-1$ and $\infty$, so by the Riemann--Hurwitz
formula $2-2\,\mathrm{genus}=2m-4(m-1)$ and the genus is $k$. Let $p_0,p_+,p_-,p_\infty$ be the points over $0,1,-1,\infty$, and put
\[
g=\frac{c_k}{w},\qquad dh=\frac{dz}{z^2-1},\qquad \Phi=\Bigl(\tfrac12(g^{-1}-g),\ \tfrac i2(g^{-1}+g),\ 1\Bigr)dh,\qquad
\x(p)=\Ree\int_{p_0}^p\Phi .
\]
For a unique $c_k>0$ the periods vanish. Then $\x$ is a proper embedding of $\overline M_k\setminus\{p_+,p_-,p_\infty\}$ onto a complete
minimal surface $M_k$ of total curvature $-4\pi(k+2)$, with catenoidal ends at $p_\pm$ and a planar end at $p_\infty$: the
Costa--Hoffman--Meeks surface of genus $k$~\cite{HoffmanMeeks1990}. The data in this form are in~\cite{Fujimori2014}, with
\[
c_k^2=\frac{2A_k}{B_k},\qquad A_k=\int_0^1\Bigl(\frac t{1-t^2}\Bigr)^{k/(k+1)}dt,\qquad B_k=\int_0^1\frac{dt}{(t^k(1-t^2))^{1/(k+1)}} ;
\]
Section~\ref{sec:num_chm} checks this formula numerically. For $k=1$ the substitution $z=\wp/e_1$, $w=\wp'/(2e_1^{3/2})$ turns the data
of Section~\ref{sec:costa} into these, with $c_1=\sqrt{2\pi/e_1}$ and $\mathfrak a\,\wp\,dz/\wp'=\frac{\mathfrak a}{4e_1}\,dh$; so $M_1$ is the Costa surface
scaled by $4e_1/\mathfrak a$.

\begin{lemma}[The central point]\label{lem:chm_center}
Near $p_0$ there is a local coordinate $t$ with $z=-t^m$ and $w=-t^{m-1}(1-t^{2m})^{1/m}$, the branch of the root being $1$ at
$t=0$. With $\ell_k=mc_k/2$ and $Z=x_1+ix_2$,
\[
Z=\ell_k\,t-\frac{m}{2(2m-1)c_k}\,\overline t^{\,2m-1}+O(|t|^{2m+1}),\qquad x_3=\Ree\operatorname{artanh}(t^m),
\]
and $|\x|=\ell_k|t|\bigl(1+O(|t|^{2k})\bigr)$. Consequently $\x(p_0)=0$, $\n(p_0)=e_3$, and near the origin $M_k$ is the graph
\[
x_3=f_k(x')=a_k\Ree(x_1+ix_2)^{k+1}+O(|x'|^{3k+1}),\qquad a_k=\ell_k^{-(k+1)} .
\]
\end{lemma}

\begin{proof}
Since $z^{m-1}(z^2-1)=(-t^m)^{m-1}(t^{2m}-1)=(-1)^mt^{m(m-1)}(1-t^{2m})$, the formulas parametrize the curve near $p_0$, and $t$ is a
local coordinate because the projection is totally ramified there. Then $g=-c_kt^{-(m-1)}(1-t^{2m})^{-1/m}$ and
$dh=mt^{m-1}(1-t^{2m})^{-1}dt$: the pole of $g$ is cancelled by the zero of $dh$, and
\begin{gather*}
\Phi_1+i\Phi_2=-g\,dh=mc_k(1-t^{2m})^{-1-1/m}dt,\\
\Phi_1-i\Phi_2=\frac{dh}g=-\frac m{c_k}\,t^{2m-2}(1-t^{2m})^{-1+1/m}dt .
\end{gather*}
Integrating from $t=0$ and using $Z=\frac12\bigl(\int(\Phi_1+i\Phi_2)+\overline{\int(\Phi_1-i\Phi_2)}\bigr)$ gives the expansions; $g(p_0)=\infty$ gives
$\n(p_0)=e_3$. Inverting, $t=Z/\ell_k+O(|Z|^{2m-1})$, so $x_3=\Ree(t^m)+O(|t|^{3m})=\ell_k^{-m}\Ree Z^m+O(|Z|^{3m-2})$, and $3m-2=3k+1$. Finally
$|\x|^2=|Z|^2+x_3^2=\ell_k^2|t|^2(1+O(|t|^{2k}))$.
\end{proof}

For $k=1$, after the scaling by $\mathfrak a/(4e_1)$, $a_1=e_1/(2\pi)$ becomes $\kappa/2$, and Lemma~\ref{lem:chm_center} becomes
Lemma~\ref{lem:saddle}. For $k\geq2$ the central point is a saddle of order $k+1$ with $|A|(p_0)=0$; the construction needs only the
regularity of the immersion there and the graph expansion. Figure~\ref{fig:chm} shows $M_2$ and $M_3$ with a hole.

\begin{figure}[ht]
\centering
\begin{subfigure}{0.36\textwidth}\includegraphics[width=\textwidth]{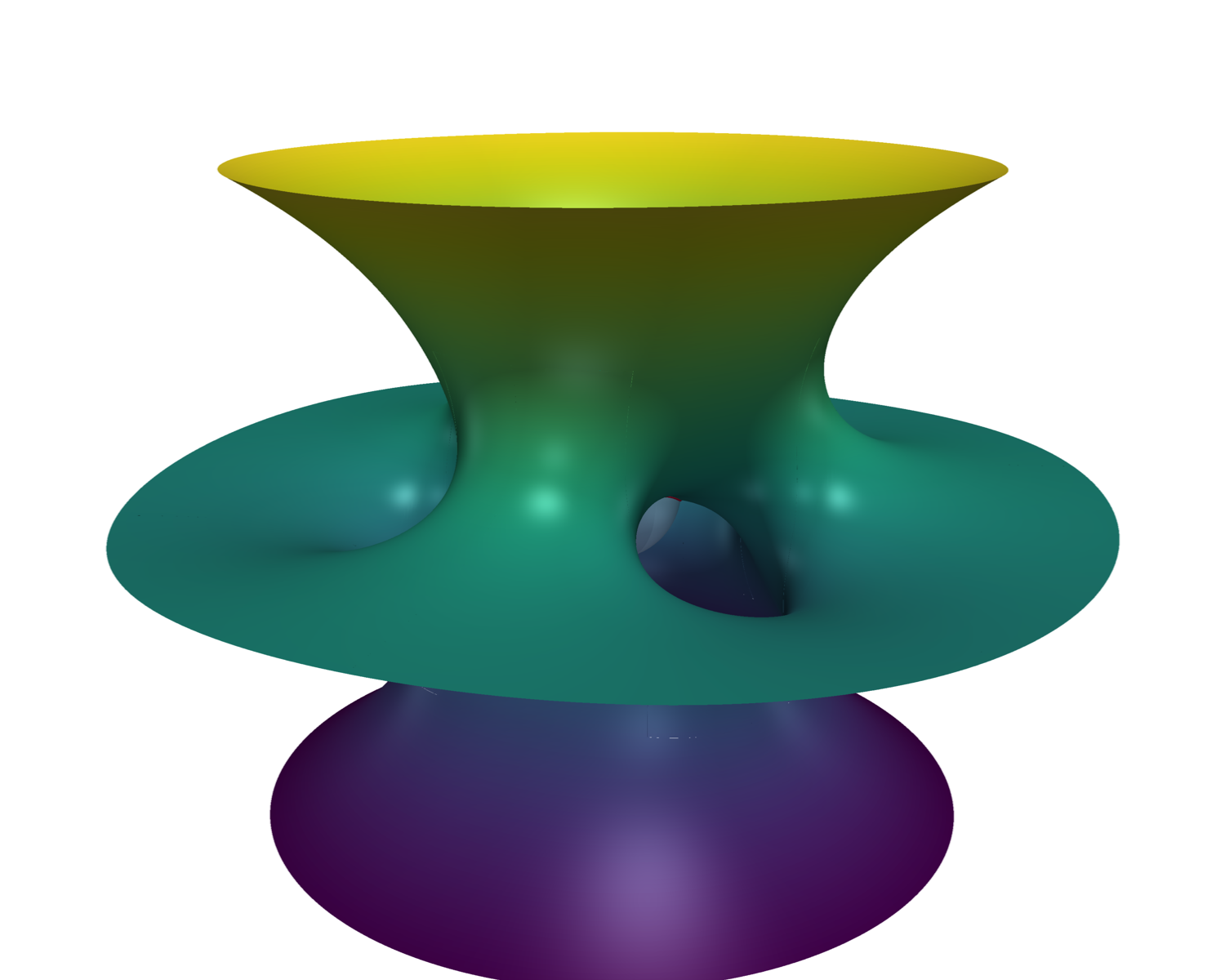}\caption{$k=2$, $\varepsilon=0.56$}\end{subfigure}
\begin{subfigure}{0.31\textwidth}\includegraphics[width=\textwidth]{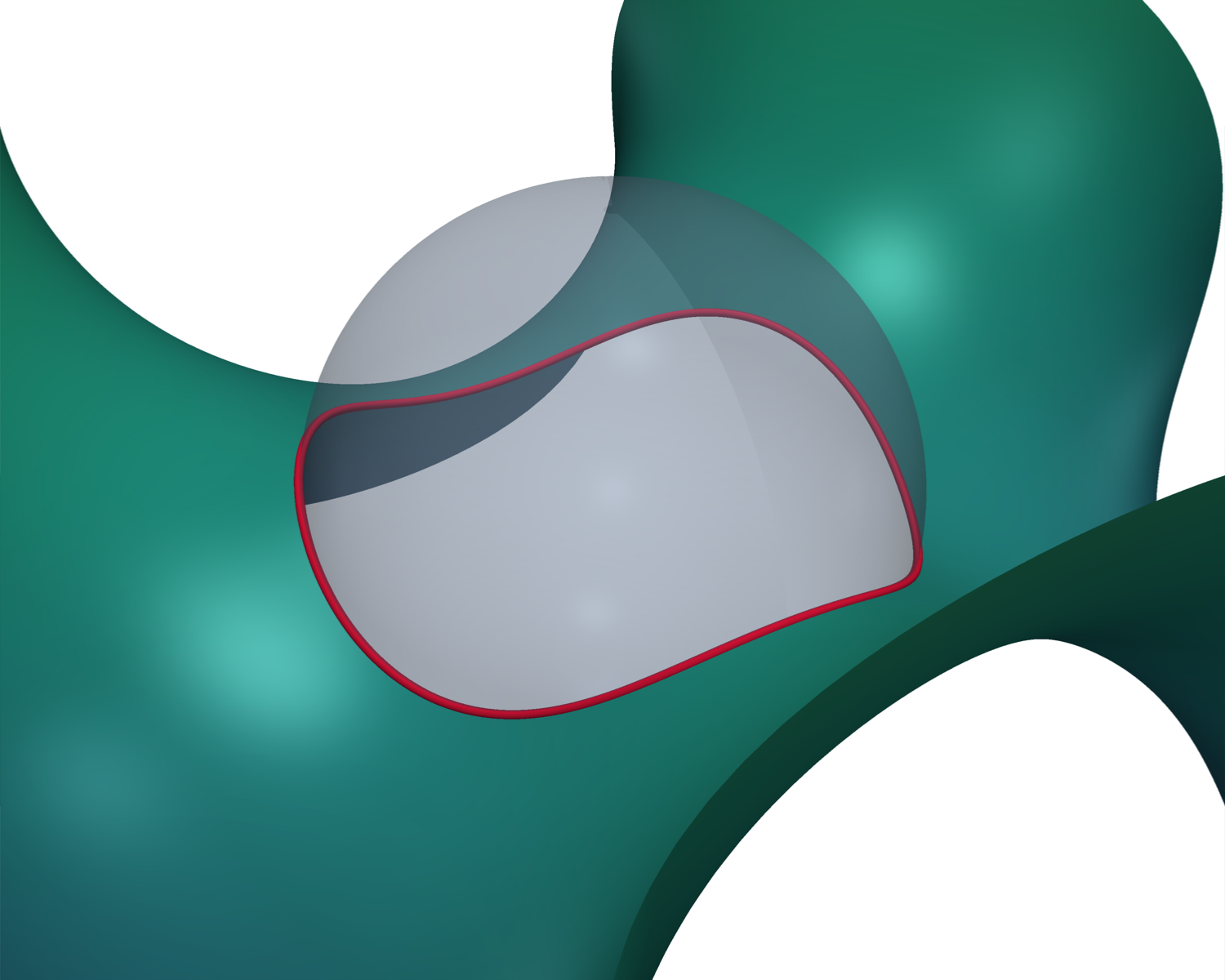}\caption{$k=2$, near the sphere}\end{subfigure}
\begin{subfigure}{0.31\textwidth}\includegraphics[width=\textwidth]{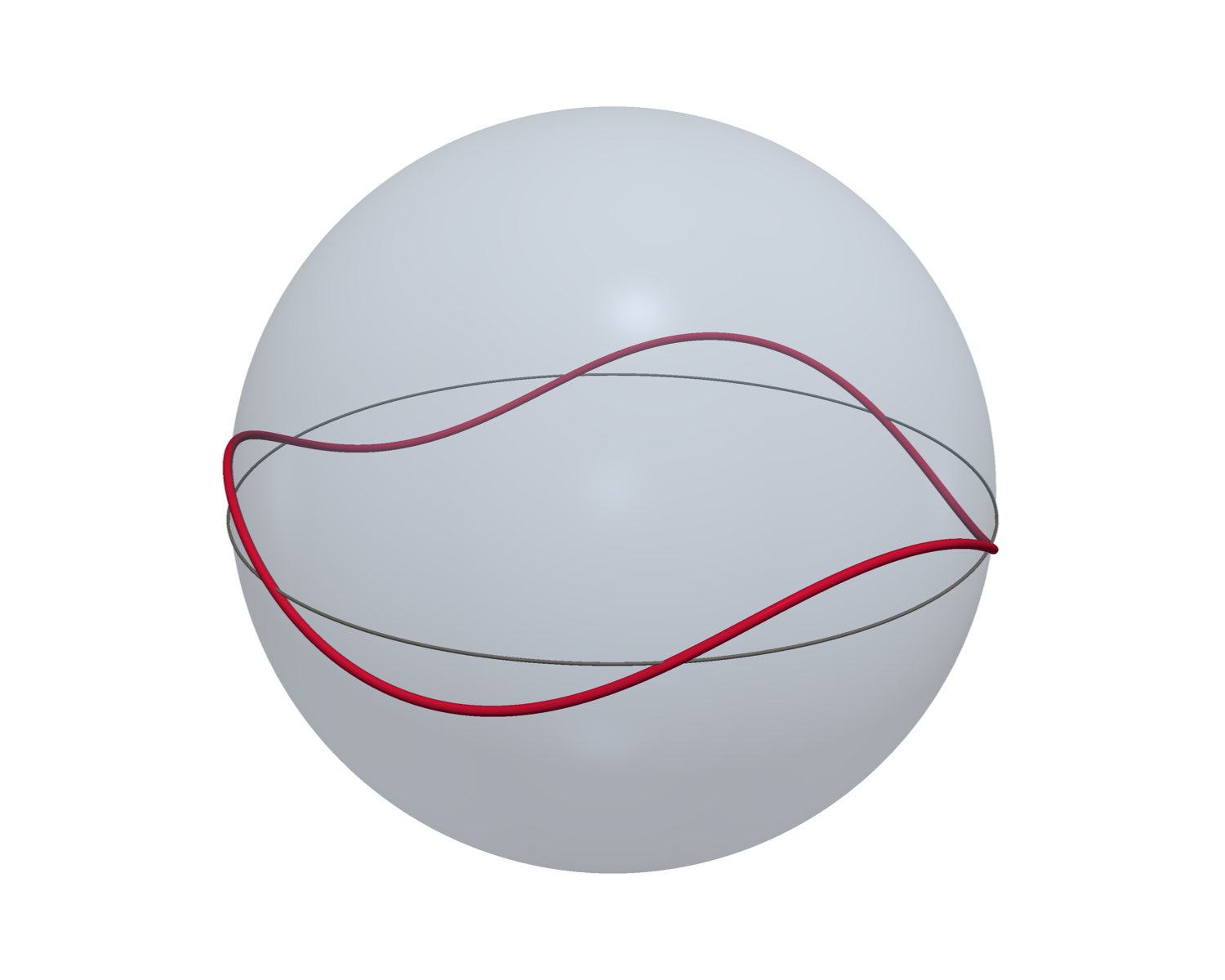}\caption{$k=2$, the boundary curve}\end{subfigure}

\medskip
\begin{subfigure}{0.36\textwidth}\includegraphics[width=\textwidth]{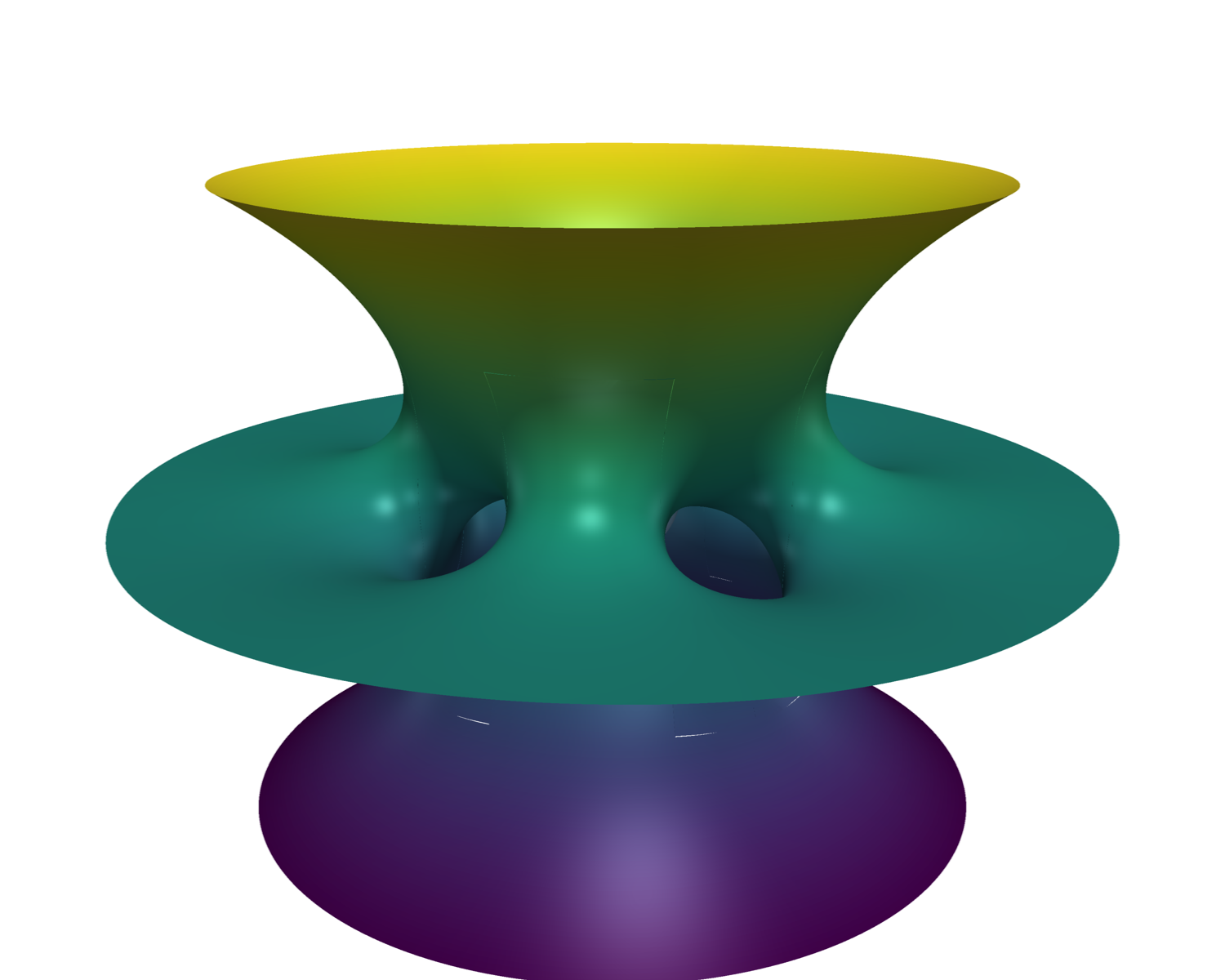}\caption{$k=3$, $\varepsilon=0.72$}\end{subfigure}
\begin{subfigure}{0.31\textwidth}\includegraphics[width=\textwidth]{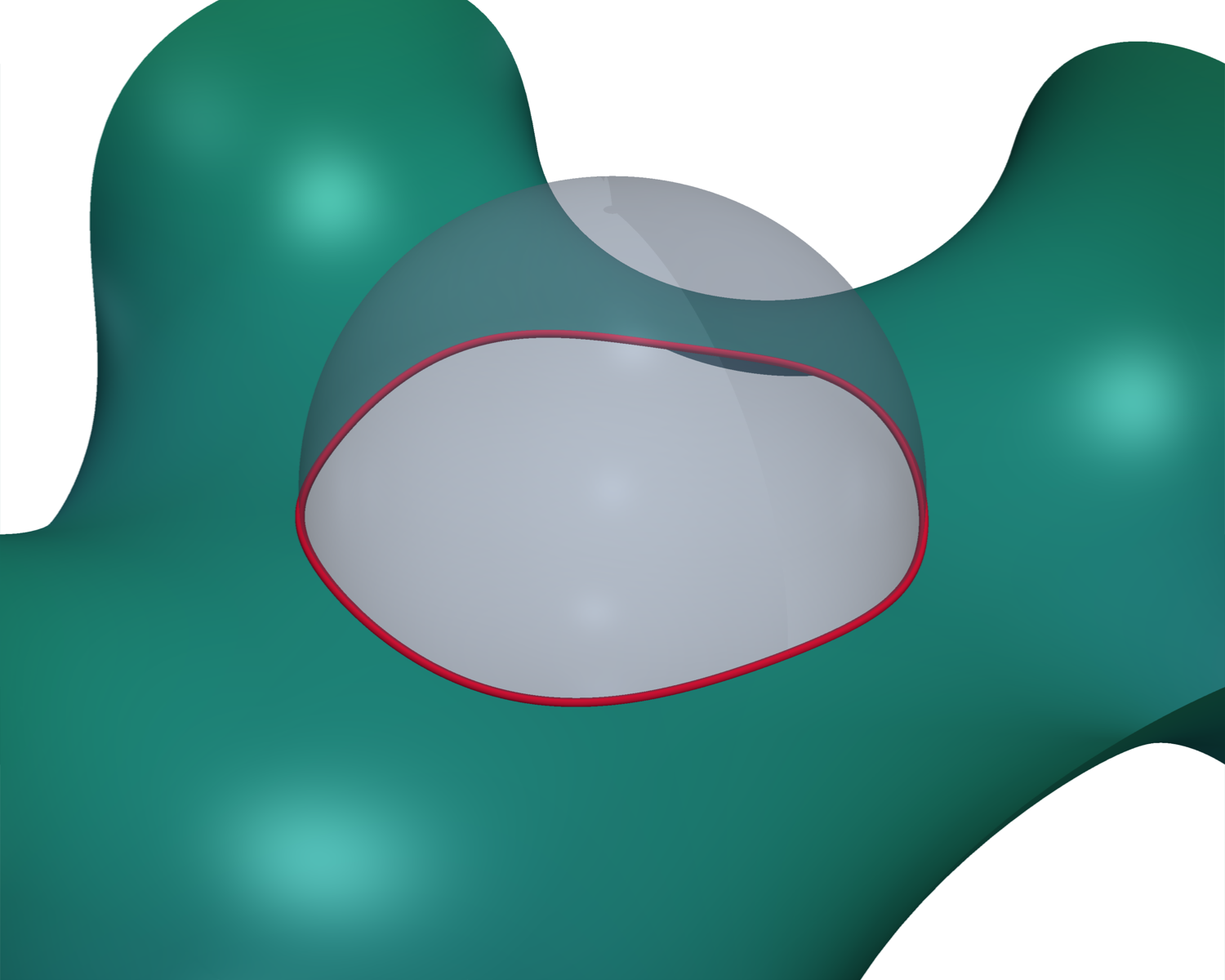}\caption{$k=3$, near the sphere}\end{subfigure}
\begin{subfigure}{0.31\textwidth}\includegraphics[width=\textwidth]{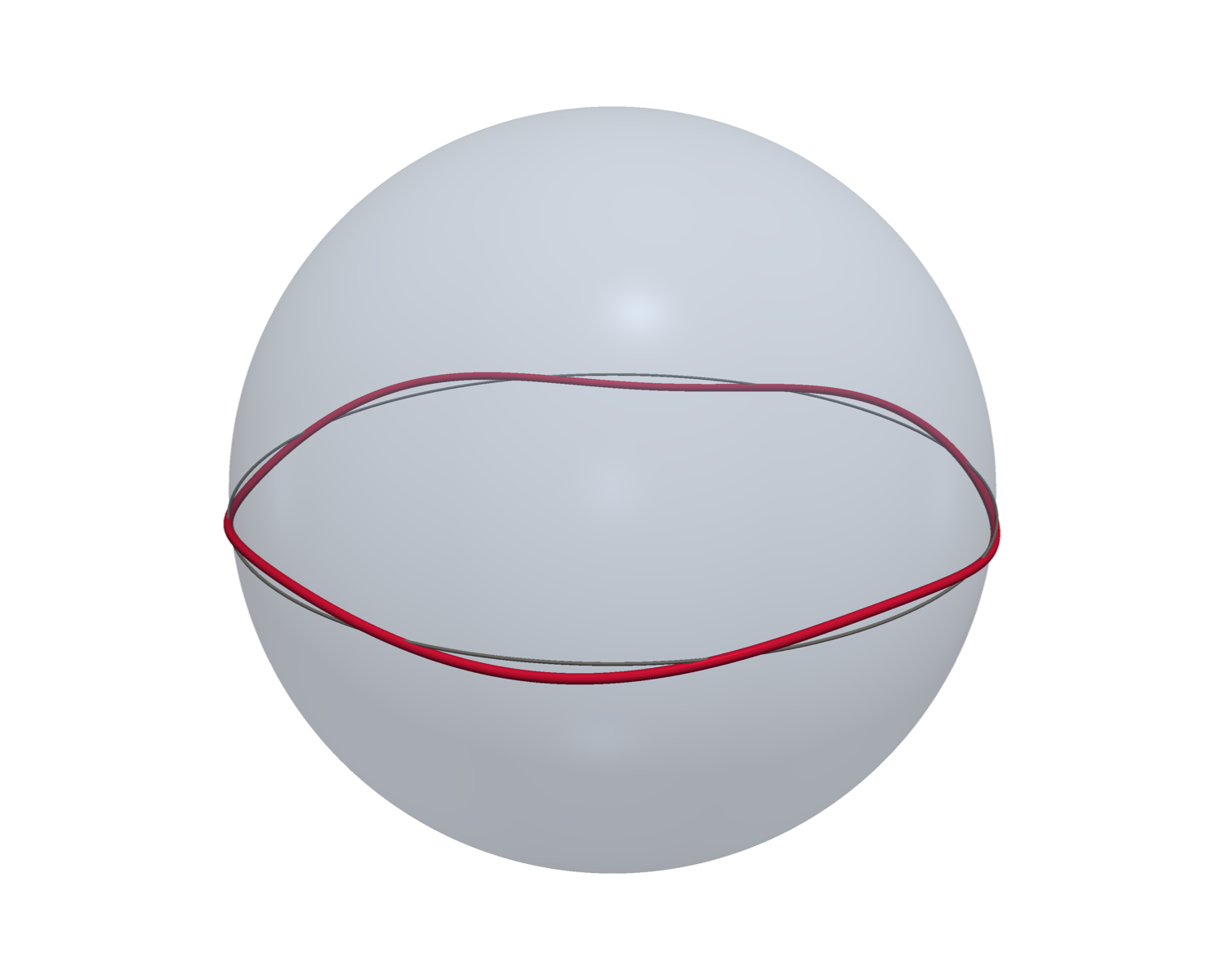}\caption{$k=3$, the boundary curve}\end{subfigure}
\caption{Numerical models of the surfaces $\Sigma_{k,\varepsilon}$ of Theorem~\ref{thm:B} for $k=2$ and $k=3$, made like Figure~\ref{fig:intro}
(Section~\ref{sec:num_chm}): the Costa--Hoffman--Meeks surface $M_k$ scaled by $\varepsilon^{-1}$, with a computed free-boundary collar near the
hole and the decaying tail of the correction farther out, cut off at $|x|=7$. In (a) and (d) the planar end is the horizontal
sheet, joined to the two catenoidal ends through openings around the axis; the hole is hidden behind the sheet, and (b)
and (e) show it. The
boundary curves (c) and (f) cross the equator at $2(k+1)$ points and reach latitude $\pm8.5^\circ$ and $\pm2.2^\circ$. Theorem~\ref{thm:B}
gives $\Sigma_{k,\varepsilon}$ only for small $\varepsilon$; for these values its existence is not proved. Interactive versions:
\viewer{chm2} and \viewer{chm3}.}
\label{fig:chm}
\end{figure}

\subsection{Symmetries}
Let $\beta=\pi/m$ and $\varrho=e^{i(m-1)\beta}$. The maps $\lambda(z,w)=(-z,\varrho w)$ and $\kappa(z,w)=(\bar z,\bar w)$ preserve the curve,
because $\varrho^m=(-1)^{m-1}$, and extend to $\overline M_k$. The first is holomorphic and the second anti-holomorphic; both fix $p_0$ and
$p_\infty$, $\lambda$ exchanges $p_+$ and $p_-$, and $\kappa$ fixes them. In the coordinate $t$ of Lemma~\ref{lem:chm_center}, $\lambda(t)=e^{i\beta}t$
and $\kappa(t)=\bar t$.

\begin{lemma}\label{lem:chm_symmetry}
$\x\circ\lambda=S\circ\x$ and $\x\circ\kappa=T\circ\x$, with $S(Z,x_3)=(e^{i\beta}Z,-x_3)$ and $T(Z,x_3)=(\bar Z,x_3)$. The group $G_k$
generated by $S$ and $T$ has order $4m=4(k+1)$, since $S^{2m}=T^2=1$ and $TST=S^{-1}$. With $\chi(\sigma)=\langle\sigma e_3,e_3\rangle$ one
has $\chi(S)=-1$, $\chi(T)=1$ and $\n(\sigma p)=\chi(\sigma)\,\sigma\n(p)$, where $\sigma\in G_k$ acts on $\overline M_k$ through $\lambda$ and $\kappa$.
A $G_k$-equivariant function, $u(\sigma p)=\chi(\sigma)u(p)$, contains near the origin, in the polar angle $\varphi$ of $Z$ or of $t$, only the
modes $\cos\bigl((2j+1)m\varphi\bigr)$, $j\geq0$, and the same holds on the planar end. On each catenoidal end only the cosines of
frequencies divisible by $m$ occur. For $k=1$, $G_1=G$, with $S=\sigma_1\tau$ and $T=\sigma_2$.
\end{lemma}

\begin{proof}
One has $\lambda^*g=\varrho^{-1}g$ and $\lambda^*dh=-dh$. Hence $\lambda^*(\Phi_1+i\Phi_2)=-\varrho^{-1}(\Phi_1+i\Phi_2)=e^{i\beta}(\Phi_1+i\Phi_2)$,
$\lambda^*(\Phi_1-i\Phi_2)=e^{-i\beta}(\Phi_1-i\Phi_2)$ and $\lambda^*\Phi_3=-\Phi_3$, that is, $\lambda^*\Phi=S\Phi$. Integrating from the fixed point $p_0$
gives $\x\circ\lambda=S\circ\x$. Since $c_k$ is real, $\kappa^*\Phi=(\overline{\Phi_1},-\overline{\Phi_2},\overline{\Phi_3})$, which gives $\x\circ\kappa=T\circ\x$. The
relations of $S$ and $T$ are checked directly. As in Lemma~\ref{lem:symmetry}, $\n\circ\lambda$ and $S\n$ are unit normals at the same
points, so they agree up to a global sign; at $p_0$ they are $e_3$ and $-e_3$, so $\n\circ\lambda=-S\n=\chi(S)S\n$. In the same way
$\n\circ\kappa=T\n$. Near the origin, $S$ acts on the angle by $\varphi\mapsto\varphi+\beta$ with $\chi=-1$, and $T$ by $\varphi\mapsto-\varphi$ with $\chi=1$. So
$u(r,\varphi+\pi/m)=-u(r,\varphi)$ and $u(r,-\varphi)=u(r,\varphi)$, which leaves exactly the cosines $\cos(l\varphi)$ with $l/m$ odd. The planar end is
fixed by $G_k$ and is a graph over the exterior of a disk in $\{x_3=0\}$, where $G_k$ acts in the same way. The stabilizer of
$p_+$ is generated by $S^2$, the rotation by $2\pi/m$ about the $x_3$-axis with $\chi=1$, and by $T$; this leaves the cosines of
frequencies divisible by $m$. Finally $S$ exchanges the two catenoidal ends.
\end{proof}

\begin{lemma}\label{lem:chm_kernel}
The bounded Jacobi fields of $M_k$ form a space of dimension $4$, spanned by $\langle\n,e_i\rangle$, $i=1,2,3$, and
$\langle\n,e_3\times\x\rangle$. No nonzero one is $G_k$-equivariant.
\end{lemma}

\begin{proof}
The dimension is $4$ by Nayatani for $k\leq37$~\cite{Nayatani1992,Nayatani1993} and by Morabito for all $k$~\cite{Morabito2009}. The four fields are
bounded Jacobi fields as in Lemma~\ref{lem:no_kernel}, since the ends of $M_k$ are regular and parallel. The translation fields are
independent, because the Gauss map extends to a map of degree $k+2$ from $\overline M_k$ onto $S^2$. If $\rho=\langle\n,e_3\times\x\rangle$ were equal
to $\langle\n,a\rangle$, the twisted action of $T$ and of $S^2T$, reflections in two different vertical planes (with $\chi=1$ and determinant
$-1$), would give $Ta=-a$ and $S^2Ta=-a$. Then $a$ would be orthogonal to both planes, so $a=0$, and $M_k$ would be rotationally
symmetric, which it is not. So the four fields span the bounded Jacobi fields. As in Lemma~\ref{lem:no_kernel}, the twisted action
$u\mapsto\chi(\sigma)\,u\circ\sigma^{-1}$ maps $\langle\n,a\rangle$ to $\langle\n,\sigma a\rangle$ and $\rho$ to $\chi(\sigma)\det(\sigma)\rho$. For $u=\langle\n,a\rangle+c\rho$
equivariant, the action of $T$ gives $c=0$, and then $\sigma a=a$ for all $\sigma\in G_k$. But $S$ has the eigenvalues $e^{\pm i\beta}$ and $-1$, none of them
equal to $1$, so $a=0$.
\end{proof}

\subsection{Proof of Theorem~\ref{thm:B}}
Sections~\ref{sec:perturb}--\ref{sec:nonlinear} use the Costa surface only through Lemmas~\ref{lem:saddle}--\ref{lem:no_kernel}, the flat metric of the
torus and the structure of the ends. We list the changes; everything else is word for word the same, with constants
depending on $k$.

\emph{The background metric.} The surface $\overline M_k$ carries no flat metric, but flatness is used only near the hole. Choose
a smooth metric in the conformal class of $\overline M_k$ equal to $|dt|^2$ on a disk $\{|t|<t_0\}$ and average it over $G_k$. The average is
conformal, because the elements of $G_k$ act holomorphically or anti-holomorphically, and it is still $|dt|^2$ on that disk, because
$G_k$ acts there by the isometries $t\mapsto e^{i\beta}t$ and $t\mapsto\bar t$. We use this metric $\tilde g$ in place of the flat metric of the torus.
Then $J=\Delta_{\tilde g}+\tilde V$ with $\tilde V=e^{2\phi}|A|^2$ smooth on $\overline M_k$~\cite{FischerColbrie1985}, and Proposition~\ref{prop:J_iso} holds by
Lemma~\ref{lem:chm_kernel}.

\emph{Section~\ref{sec:perturb}.} The sets $U$, $D_\varepsilon$, $\Omega_\varepsilon$ and the field $W$ are defined in the same way: $g=\infty$ at $p_\pm$ and $g=0$ at
$p_\infty$, so $\varsigma=+1$ on the slabs of the catenoidal ends and $\varsigma=-1$ on that of the planar end, and $W(\sigma x)=\chi(\sigma)\sigma W(x)$ because
$S$ exchanges the catenoidal slabs. By Lemma~\ref{lem:chm_center}, $\theta=1+O(|x|^{2k})$ near the origin. Lemma~\ref{lem:error} becomes
\[
\Theta_\varepsilon(0)=-k\,a_k\,\varepsilon^k\cos\bigl((k+1)\varphi\bigr)+O(\varepsilon^{3k}),
\]
since $\langle\n,x\rangle=(f_k-x'\cdot\nabla f_k)/\sqrt{1+|\nabla f_k|^2}$ and the leading term of $f_k$ is homogeneous of degree $k+1$. Lemma~\ref{lem:linearization}
holds as stated; here $W^T=O(|x|^k)$ and $|A|=O(|x|^{k-1})$ near the origin give even $|R_\varepsilon u|\leq C\varepsilon^{2k-1}|u|$.

\emph{Section~\ref{sec:linear}.} The weight is $\vartheta=|\x|$ near $p_0$, and the rescaled variable is $y=\ell_kt/\varepsilon$ in place of $\pi w/\varepsilon$; by
Lemma~\ref{lem:chm_center}, $|\x|=\ell_k|t|(1+O(|t|^{2k}))$ and $e^\phi=\ell_k(1+O(|t|^{2k}))$. In the three limiting problems of Lemma~\ref{lem:limits} the group
acts by $y\mapsto e^{i\beta}y$ with $\chi=-1$ and by $y\mapsto\bar y$ with $\chi=1$, so their solutions have only the modes $(2j+1)m\geq2$. The proof of
Lemma~\ref{lem:limits} applies: the growth $O(r)$ excludes $r^{(2j+1)m}$, the Robin condition gives $\bigl((2j+1)m+1\bigr)b=0$, and boundedness near
the origin excludes $r^{-(2j+1)m}$ in the punctured plane. Proposition~\ref{prop:uniform} follows.

\emph{Section~\ref{sec:nonlinear}.} Lemma~\ref{lem:ends} applies at all three ends. At $p_+$ take $\zeta$ with $z-1=\zeta^m$; then
$w=\zeta F(\zeta^m)$ with $F$ holomorphic and $F(0)=2^{1/m}$, so $g\,dh=c_km\,\zeta^{-2}G(\zeta^m)\,d\zeta$ with $G$ holomorphic, $g^{-1}dh$ is holomorphic,
and $dh$ has a simple pole with the real residue $m/2$. A term $\zeta^{-1}$ in $g\,dh$ would need $mj=1$, so the horizontal forms have no
residues, the inversion $y=x'/|x'|^2$ is a smooth coordinate at $p_+$, and $x_3=a\log|x'|+b+h(y)$ with $h$ smooth; the same holds at
$p_-$. At $p_\infty$ take $z=\zeta^{-m}$; then $g=c_k\zeta^{m+1}(1-\zeta^{2m})^{-1/m}$, $dh=-m\zeta^{m-1}(1-\zeta^{2m})^{-1}d\zeta$ and
$g^{-1}dh=-\frac m{c_k}\zeta^{-2}(1-\zeta^{2m})^{1/m-1}d\zeta$, which again has no residue, and $dh$ is holomorphic: a planar end. Near the hole the
rescaled pieces of $M_k$ are graphs with $C^{3,\alpha}$ norm $O(\vartheta^k)$, which is only better than in Lemma~\ref{lem:nonlinear}.

\emph{The fixed point.} Since $\|\Theta_\varepsilon(0)\|_Z\leq C\varepsilon^k$, the contraction argument in the proof of Theorem~\ref{thm:A} gives a $G_k$-equivariant
solution with $\|u\|_X\leq C\varepsilon^k$. The arguments for the free boundary, the exteriority, the embeddedness and the ends are unchanged, with
Lemma~\ref{lem:chm_symmetry} in place of Lemma~\ref{lem:modes} at the planar end.
The parameter domain is $\overline M_k$ minus a disk and three points, so $\Sigma_{k,\varepsilon}$ has genus $k$, one boundary circle and three ends, and it is
unstable by~\cite[Theorem~1.1]{MazetMendes2022}.

\emph{The index.} The Costa--Hoffman--Meeks surface $M_k$ has index $2k+3$~\cite{Nayatani1992,Nayatani1993,Morabito2009}. The proof of
Proposition~\ref{prop:index} applies with a $(2k+3)$-dimensional negative space and gives $\operatorname{ind}\Sigma_{k,\varepsilon}\geq2k+3$. This finishes the proof
of Theorem~\ref{thm:B}. \qed

\begin{remark}[The leading correction]
The model problem of Section~\ref{sec:num_collar} has the analogue $\Delta u=0$ outside the hole and
$-\partial_ru+u/\varepsilon=k\,a_k\,\varepsilon^k\cos\bigl((k+1)\varphi\bigr)$ on $|x|=\varepsilon$. Its decaying solution is
\[
u=\frac{k\,a_k}{k+2}\,\frac{\varepsilon^{2k+2}}{|x|^{k+1}}\cos\bigl((k+1)\varphi\bigr),
\]
of size $\frac{ka_k}{k+2}\varepsilon^{k+1}$ at the boundary. For $k=1$, with $a_1=\kappa/2$ after the scaling, this is~\eqref{eq:model}. The correction becomes
smaller with the genus, because the central point of $M_k$ is a saddle of order $k+1$.
\end{remark}

\section{Remarks}\label{sec:remarks}

\begin{proposition}[The index]\label{prop:index}
For all sufficiently small $\varepsilon$, the Morse index of $\Sigma_\varepsilon$ is at least $5$.
\end{proposition}

\begin{proof}
Here the index is the number of negative directions of the second variation
$Q(u)=\int_\Sigma(|\nabla u|^2-|A|^2u^2)\,dA+\int_{\partial\Sigma}u^2\,ds$ on compactly supported functions. By
Nayatani~\cite{Nayatani1993} the Costa surface has index $5$. Since the index of a complete minimal surface is attained on compact
domains~\cite{FischerColbrie1985}, there is a five-dimensional space $V$ of smooth compactly supported functions on $M$ on which
$Q_M(u)=\int_M(|\nabla u|^2-|A|^2u^2)\,dA$ is negative definite. Let $\eta_\delta$ be the logarithmic cut-off: $\eta_\delta=0$ on $\{|x|\leq\delta^2\}$,
$\eta_\delta=\log(|x|/\delta^2)/\log(1/\delta)$ on $\{\delta^2\leq|x|\leq\delta\}$ and $\eta_\delta=1$ outside $B_\delta$. Near the origin $M$ is a graph over
the plane, so $\int_M|\nabla\eta_\delta|^2\,dA\leq C/\log(1/\delta)\to0$: a point has zero capacity in dimension two. Since $V$ is finite-dimensional,
$Q_M(\eta_\delta u)\to Q_M(u)$ uniformly for $u$ in the unit sphere of $V$, so for a fixed small $\delta$ the form $Q_M$ is negative definite on
the five-dimensional space $\eta_\delta V$. Its functions are supported in a fixed compact set $K\subset M\setminus B_{\delta^2}$. By
Theorem~\ref{thm:A}(iii), $\varepsilon\Sigma_\varepsilon\to M$ smoothly on compact subsets of $\R^3\setminus\{0\}$. So for small $\varepsilon$ a neighbourhood of $K$
in $\varepsilon\Sigma_\varepsilon$ is a normal graph over a neighbourhood of $K$ in $M$. Transplanting $\eta_\delta V$ along this graph gives a five-dimensional
space on which the second variation of $\varepsilon\Sigma_\varepsilon$ converges uniformly to $Q_M$, hence is negative definite. These functions vanish
near the boundary of $\varepsilon\Sigma_\varepsilon$, which lies on $\partial B_\varepsilon$ inside $B_{\delta^2}$, so the boundary term does not enter. Scaling does not
change the index.
\end{proof}

\begin{remark}[Jacobi fields of $\Sigma_\varepsilon$]
The rotations about the $x_3$-axis preserve the unit ball, so $\langle\n,e_3\times\x\rangle$ is a bounded Jacobi field of $\Sigma_\varepsilon$ which
satisfies the Robin condition of the free boundary. It is not zero, because $\Sigma_\varepsilon$ is not rotationally symmetric. So
$\Sigma_\varepsilon$ is degenerate in the full, non-symmetric class, with nullity at least one. By Proposition~\ref{prop:index} its
index is at least $5$. The spectrum of Section~\ref{sec:num_spectrum} suggests that one negative direction is $G$-equivariant;
this is not proved.
\end{remark}

\begin{problem}
What is the exact index of $\Sigma_\varepsilon$? Is $\Sigma_\varepsilon$ nondegenerate among $G$-invariant surfaces? Are there connected embedded exterior
FBMS with more than one planar end?
\end{problem}

\section{Numerical experiments}\label{sec:numerics}

None of the proofs above uses a computation. This section describes the experiments that check the ingredients of
the proof and the computation behind the pictures. The scripts are listed in Section~\ref{sec:num_scripts}, and the
values quoted below are those they print; Tables~\ref{tab:defect}--\ref{tab:collar} are copied from that output.

\subsection{Two implementations of the Costa surface}\label{sec:num_impl}
\emph{(a) Jacobi elliptic functions.} For the square lattice $\Z+i\Z$ one has $e_2=0$, $e_3=-e_1$ and
\[
\wp(z)=-e_1+\frac{2e_1}{\operatorname{sn}^2(\sqrt{2e_1}\,z)},\qquad
\wp'(z)=-2(2e_1)^{3/2}\,\frac{\operatorname{cn}\operatorname{dn}}{\operatorname{sn}^3}(\sqrt{2e_1}\,z),
\]
with the parameter $m=\frac12$ and $e_1=2K(\frac12)^2=6.875185818\ldots$, where $K(\frac12)$ is the complete elliptic integral of the
first kind with parameter $\frac12$, so that the real half-period is $\frac12$. The functions of a complex argument are obtained from those
of real arguments with parameters $m$ and $1-m$ by the addition formulas. The position $\x(z)=\Ree\int_\omega^z\Phi$ is
integrated by the trapezoidal rule along rays from $\omega$ to the boundary of the fundamental square: $1441$ rays with
$3000$ points each for the checks, and $900$ rays with $1600$ points for the pictures.

\emph{(b) Theta series.} An independent implementation, written by an auditing agent, computes $\wp$ and $\wp'$ from the
Jacobi theta functions with nome $q=e^{-\pi}$ (twelve terms) and integrates $\Phi$ by Gauss--Legendre quadrature on straight
segments from $\omega$. It is used for the Jacobi operator in Section~\ref{sec:num_spectrum}.

The two implementations give the same $e_1$ to $9\cdot10^{-16}$. At thirty random points of the fundamental square
their positions $\x(z)$ differ by at most $3\cdot10^{-9}$, which is the error of the trapezoidal rule, and their unit
normals by at most $2\cdot10^{-15}$.

\subsection{The period condition and the constants}\label{sec:num_period}
In $30$-digit arithmetic, the horizontal period condition $\oint g\,dh=\overline{\oint dh/g}$ on the cycle $\Imm z=\frac14$ gives
$\mathfrak a^2=1187.97893710448$, which agrees with $8\pi e_1^2$ in all fifteen printed digits. Hence
\[
\mathfrak a=2\sqrt{2\pi}\,e_1=34.46707033,\qquad \kappa=\frac{\mathfrak a}{2\pi^2}=1.746122181,\qquad \frac\kappa2=0.873061 .
\]

\subsection{Symmetries}\label{sec:num_symmetry}
At forty random points of the fundamental square, the identities $\x\circ\kappa_\sigma=\sigma\circ\x$ of
Lemma~\ref{lem:symmetry} hold with errors at most $1.5\cdot10^{-14}$ for $\sigma_2$, $3.0\cdot10^{-15}$ for $\sigma_1$ and
$4.7\cdot10^{-11}$ for $\tau$ and the rotoreflection $S$; the larger error comes from the quadrature along the longer
paths. The signs are $\chi(\sigma_2)=+1$ and $\chi(\tau)=-1$, and $\n(\omega)=e_3$ to $6\cdot10^{-9}$.

\subsection{The ball meets the surface in a disk}\label{sec:num_disk}
Section~\ref{sec:perturb} uses that $\x^{-1}(B_\rho)$ is a disk around $\omega$ for small $\rho$; this follows from
Lemma~\ref{lem:saddle} and needs no computation. Numerically the property holds for every $\rho<\rho_*=1.8829$. Along each of the $1441$ rays from $\omega$ to the boundary of the fundamental square,
$|\x|$ increases strictly until it reaches a value at least $\rho_*$, and on the rest of every ray $|\x|\geq\rho_*$. So for
$\rho<\rho_*$ the set $\{|\x|<\rho\}$ is star-shaped in the coordinate $z$. It is a disk, and $|\x|$ has no critical point in it
other than $\omega$. The rays are sampled at $3000$ points each and stop at $0.999$ of the distance to the boundary of the
square, so this is a numerical observation and not a rigorous bound: a certified threshold would need estimates between the
rays and in the thin strip near the boundary. The level curves are shown in Figure~\ref{fig:torus}. The proof needs this only for small $\rho_0$, but the
pictures with $\varepsilon=0.45$ rely on the larger range.

\subsection{The boundary error}\label{sec:num_defect}

\begin{figure}[ht]
\centering
\includegraphics[width=\textwidth]{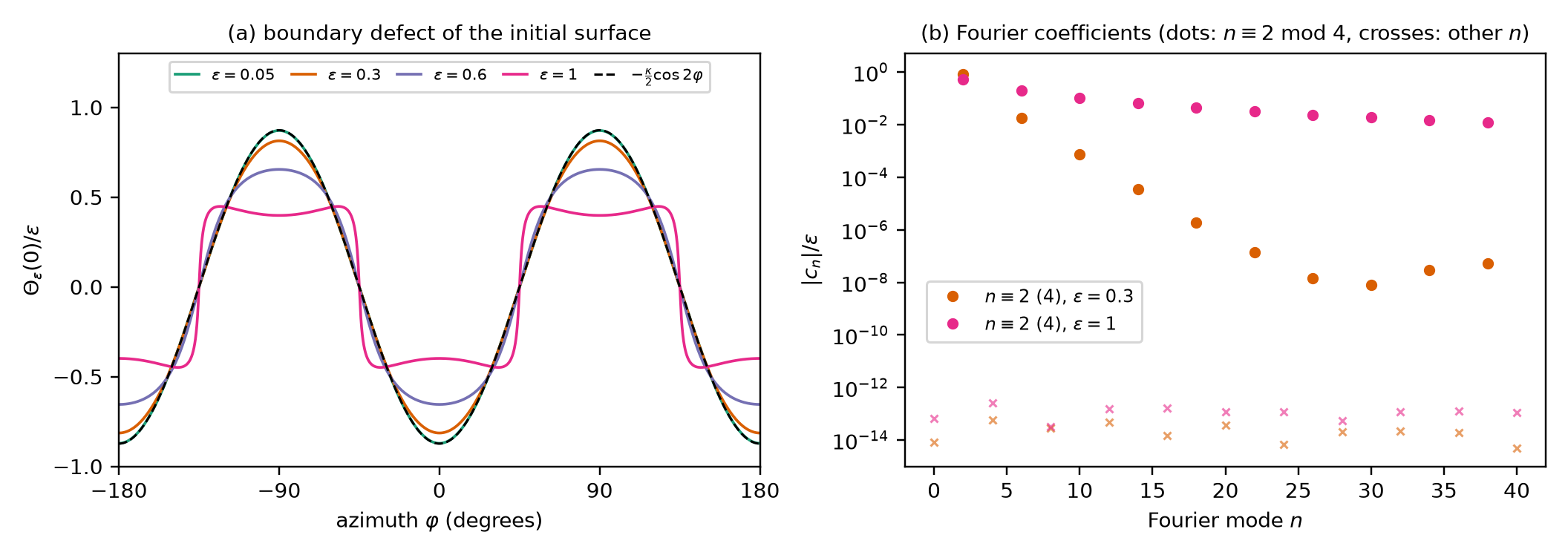}
\caption{The boundary error $\Theta_\varepsilon(0)=\langle\n,x/|x|\rangle$ of the initial surface $M\setminus B_\varepsilon$ on the circle
$|x|=\varepsilon$. (a) Divided by $\varepsilon$, against the azimuth, compared with the leading term $-\frac\kappa2\cos2\varphi$ of
Lemma~\ref{lem:error}. For $\varepsilon=0.05$ the two curves coincide to plotting accuracy; for $\varepsilon=1$ the expansion is no longer
useful. (b) The Fourier coefficients. Only the modes $n\equiv2\pmod4$ occur (dots); all others (crosses) are at the level of
rounding errors, as Lemma~\ref{lem:modes} predicts for the $G$-equivariant function $\Theta_\varepsilon(0)$. For $\varepsilon=0.3$ the
coefficients decay geometrically until they reach the interpolation error, about $10^{-8}$.}
\label{fig:defect}
\end{figure}

For each $\varepsilon$ we locate on every ray the point with $|\x|=\varepsilon$ by linear interpolation, evaluate the exact normal
there and compute $\Theta_\varepsilon(0)=\langle\n,\x/|\x|\rangle$ and its Fourier coefficients $c_n$ in the azimuth $\varphi$ of
$\x$. Table~\ref{tab:defect} and Figure~\ref{fig:defect} show the result.

\begin{table}[ht]
\centering
\begin{tabular}{cccccc}
\toprule
$\varepsilon$ & $\max|\Theta_\varepsilon(0)|/\varepsilon$ & $|c_2|/\varepsilon$ & $|c_6|/\varepsilon$ & $|c_{10}|/\varepsilon$ & $\max_{n\not\equiv2\,(4)}|c_n|/\varepsilon$\\
\midrule
$0.05$ & $0.8714$ & $0.8718$ & $4.2\cdot10^{-4}$ & $9.8\cdot10^{-7}$ & $3\cdot10^{-13}$\\
$0.30$ & $0.8139$ & $0.8314$ & $1.8\cdot10^{-2}$ & $7.3\cdot10^{-4}$ & $6\cdot10^{-14}$\\
$0.60$ & $0.6548$ & $0.7289$ & $8.8\cdot10^{-2}$ & $1.7\cdot10^{-2}$ & $1\cdot10^{-13}$\\
$1.00$ & $0.4485$ & $0.5293$ & $1.96\cdot10^{-1}$ & $1.04\cdot10^{-1}$ & $3\cdot10^{-13}$\\
\bottomrule
\end{tabular}
\caption{The boundary error of the initial surface. The leading coefficient tends to $\kappa/2=0.8731$, as in
Lemma~\ref{lem:error}, and the coefficient of $\cos6\varphi$ is relatively of order $\varepsilon^2$.}
\label{tab:defect}
\end{table}

At $\varepsilon=0.05$ the maximum of $|\Theta_\varepsilon(0)|/\varepsilon$ is $0.8714$, against $\kappa/2=0.8731$. The relative difference
$0.2\%$ is of the size $\varepsilon^2$ of the next term in Lemma~\ref{lem:error}. The relative size of $c_6$ grows from
$4.8\cdot10^{-4}$ to $2.2\cdot10^{-2}$ between $\varepsilon=0.05$ and $0.3$, a factor $45$ for a factor $36$ in $\varepsilon^2$. The angle
between the surface and the sphere differs from a right angle by $\arcsin\max|\Theta_\varepsilon(0)|$. This is $2.5^\circ$ for
$\varepsilon=0.05$ and $19.5^\circ$ for $\varepsilon=0.45$.

\subsection{The Jacobi operator on the torus}\label{sec:num_spectrum}

\begin{figure}[ht]
\centering
\includegraphics[width=\textwidth]{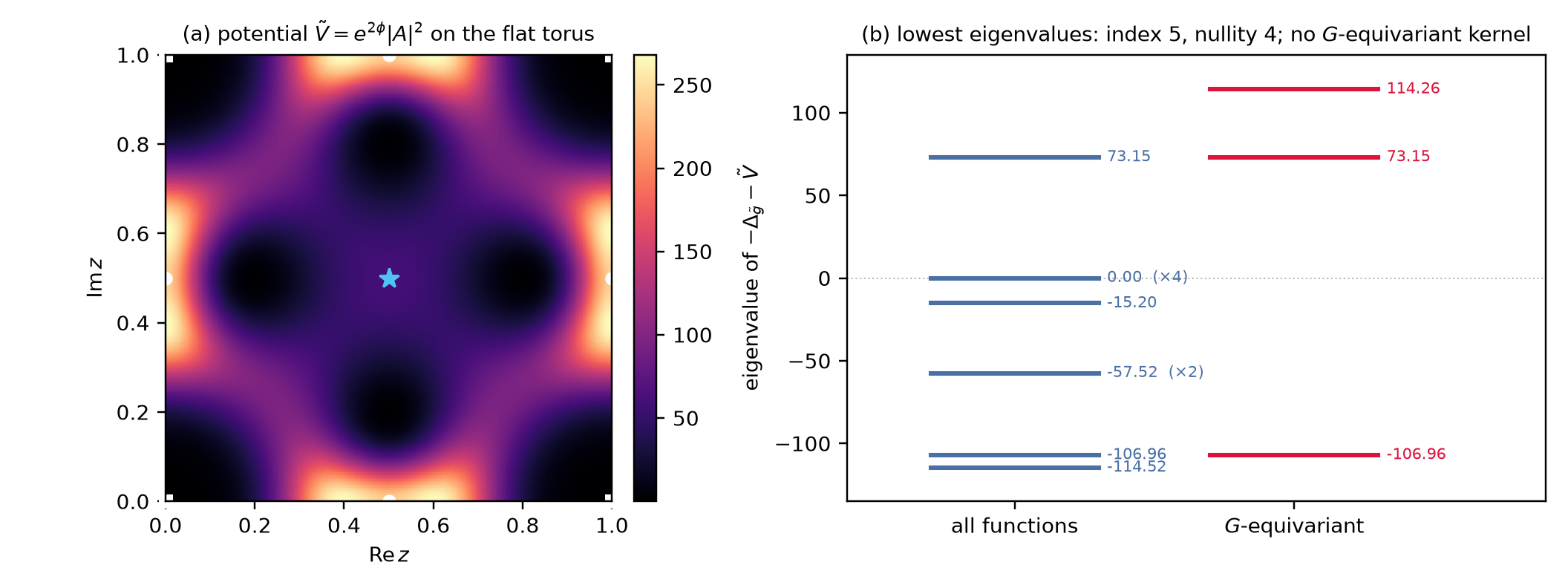}
\caption{(a) The potential $\tilde V=e^{2\phi}|A|^2$ of the compactified Jacobi operator $J=\Delta_{\tilde g}+\tilde V$ on the flat
torus. It is smooth everywhere, also at the ends: it vanishes at the planar end (corners) and is largest near the
catenoidal ends (midpoints of the sides). (b) The lowest eigenvalues of $-J$ on all functions (left) and on the
$G$-equivariant functions (right), with multiplicities. There are five negative eigenvalues and a four-dimensional kernel.
The kernel has no $G$-equivariant part: on the equivariant functions the spectrum jumps from $-106.96$ to $73.15$.}
\label{fig:spectrum}
\end{figure}

By Section~\ref{sec:linear}, $L=e^{-2\phi}J$ with $J=\Delta_{\tilde g}+\tilde V$ on the flat torus $\overline T$, and in the coordinate $z$
\[
\tilde V=\frac{8|g'|^2}{(1+|g|^2)^2}=\frac{8\mathfrak a^2|\wp''|^2}{\bigl(|\wp'|^2+\mathfrak a^2\bigr)^2},\qquad \wp''=6\wp^2-2e_1^2 .
\]
At the planar end $\wp'\sim-2z^{-3}$ and $\wp''\sim6z^{-4}$, so $\tilde V=O(|z|^4)$. At a catenoidal end $\wp'=0$ and
$\tilde V=8|\wp''|^2/\mathfrak a^2$ is finite. We discretize $-J$ by the Fourier pseudo-spectral method on an $N\times N$ grid of the
unit square, shifted by half a cell so that it avoids the ends, and compute all eigenvalues of the resulting symmetric
matrix. As a check, the discrete integral of $\tilde V$ is $75.3982236862$. This equals $24\pi$, because
$\int\tilde V\,dA_{\tilde g}=\int_M|A|^2\,dA=-2\int_MK\,dA=24\pi$.

The lowest eigenvalues of $-J$ for $N=40$ and $N=64$ agree in all printed digits:
\[
-114.518,\ -106.960,\ -57.516\ (\times2),\ -15.199,\ 0\ (\times4),\ 73.148 .
\]
So the index is $5$ and the nullity is $4$. This agrees with Nayatani's computation~\cite{Nayatani1993} and with
Theorem~\ref{thm:nayatani}, since by Section~\ref{sec:linear} the kernel of $J$ consists of the bounded Jacobi fields of $M$.

The group $G$ acts on the grid by the permutations induced by $\kappa_1,\kappa_2,\iota$, with the signs $\chi$. The eight
resulting operators commute with the matrix to $3\cdot10^{-10}$. The projections of the four kernel vectors onto the
$G$-equivariant functions vanish, which is Lemma~\ref{lem:no_kernel}. On the $G$-equivariant subspace (of dimension $190$
for $N=40$ and $496$ for $N=64$) the lowest eigenvalues are
\[
-106.960,\quad 73.148,\quad 114.261,\quad 284.406,\ \ldots
\]
So $J$ is invertible on $G$-equivariant functions with a spectral gap of $73$ around zero. This is the numerical
counterpart of Proposition~\ref{prop:J_iso}. The one negative equivariant eigenvalue is the second eigenvalue of
$-J$ on all functions. It suggests that $\Sigma_\varepsilon$ is unstable already among $G$-invariant variations, but we have not
proved this.

\subsection{The free-boundary correction near the hole}\label{sec:num_collar}

\begin{figure}[ht]
\centering
\includegraphics[width=\textwidth]{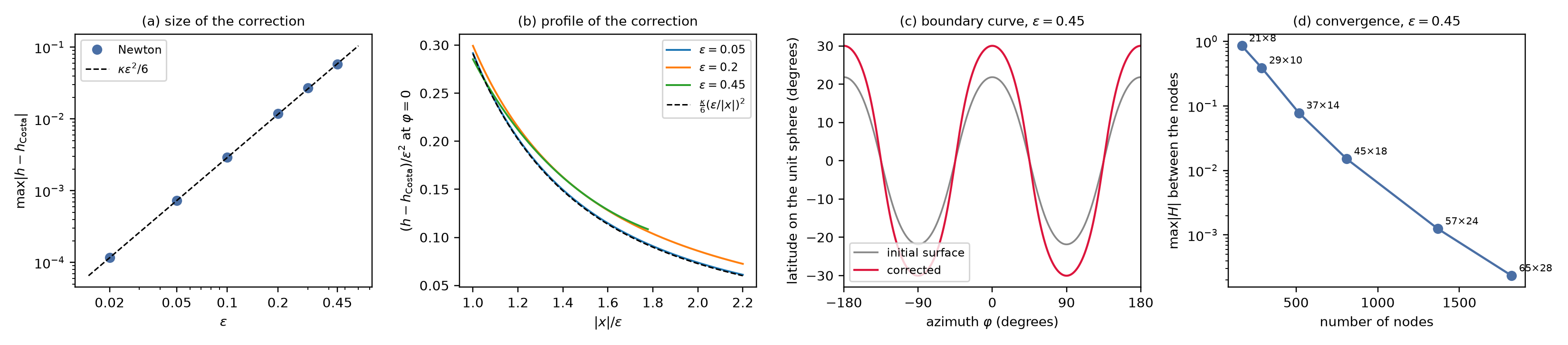}
\caption{The free-boundary collar computed by Newton's method. (a) The maximal correction $\max|h-h_{\rm Costa}|$ against
$\varepsilon$, compared with $\kappa\varepsilon^2/6$ from the model problem. (b) The correction along the ray $\varphi=0$, rescaled
by $\varepsilon^2$, against $|x|/\varepsilon$; the curves approach the model profile $\frac\kappa6(\varepsilon/|x|)^2$ as $\varepsilon\to0$. (c) The
latitude of the boundary curve on the rescaled unit sphere for $\varepsilon=0.45$: the corrected surface (red) meets the
sphere along a curve reaching $\pm30.0^\circ$, the initial surface (grey) along a curve reaching $\pm21.8^\circ$. Both cross the
equator at the azimuths $45^\circ+90^\circ k$. (d) The largest mean curvature of the computed collar between the collocation
nodes, for $\varepsilon=0.45$, against the number of nodes (Table~\ref{tab:convergence}).}
\label{fig:collar}
\end{figure}

\emph{The unknowns.} Near the hole we write the surface by heights on the spheres $\partial B_s$,
\[
X(s,\varphi)=\bigl(q\cos\varphi,\ q\sin\varphi,\ h(s,\varphi)\bigr),\qquad q=\sqrt{s^2-h^2},\qquad \varepsilon\leq s\leq S_1 .
\]
This is the ansatz of Section~\ref{sec:perturb}: the flow of $W$ near the origin moves points along the meridians of the
spheres $\partial B_s$ and keeps $s$ and the azimuth. So changing $h$ at fixed $(s,\varphi)$ is the deformation
$\Sigma_\varepsilon(u)$ there. For the Costa surface itself $h=h_{\rm Costa}(s,\varphi)$ is single-valued as long as the level curves of
$|\x|$ on $M$ are graphs over the azimuth. Numerically they are for $s\leq1.05$ and fold over for $s\gtrsim1.1$. By
Lemma~\ref{lem:modes} we expand $h=\sum_{n\equiv2\,(4)}h_n(s)\cos n\varphi$. The heights $h_{\rm Costa}$ at the nodes are computed
without interpolation: for each node $(s,\varphi)$ we find the point $z$ of the torus with $|\x(z)|=s$ and azimuth $\varphi$ by a damped
Newton method in $z$, with $\x(z)$ computed by Gauss--Legendre quadrature, to $10^{-11}$.

\emph{The equations.} Mean curvature zero at the interior collocation points. Orthogonality $\langle\n,X\rangle=0$ at $s=\varepsilon$.
In these coordinates $\langle X_s\times X_\varphi,X\rangle=s\,(h-s\,h_s)$, so this condition is $h=s\,h_s$; it is linear in $h$. At $s=S_1$
the decay condition $\partial_s(h-h_{\rm Costa})_n=-\frac ns(h-h_{\rm Costa})_n$ for each mode, which selects the decaying harmonic $s^{-n}$
of the flat model and lets the correction continue beyond $S_1$.

\emph{The discretization.} We use $65$ Chebyshev--Lobatto points in $s$ and the $28$ modes $n=2,6,\dots,110$, collocated at
$28$ azimuths in $(0,\frac\pi4)$; the other azimuths follow from the symmetry. The derivatives are spectral. Newton's method
starts from $h_{\rm Costa}$, uses a finite-difference Jacobian and stops at the level of rounding errors, about $10^{-9}$, after six or
seven iterations. We take $S_1=\min(4\varepsilon,0.8)$.

\emph{Verification between the nodes.} A small residual at the nodes does not by itself control the geometry between
them. We therefore evaluate the computed surface, with exact derivatives of its spectral representation, on a grid of
$161\times181$ points of $[\varepsilon,S_1]\times[0,\frac\pi4]$, and compute its mean curvature and the angle at the boundary there. The
same evaluator gives $|H|\leq3\cdot10^{-13}$ on a catenoid and $10^{-15}$ on a plane. Since the boundary condition $h=s\,h_s$ is linear
and holds at as many azimuths as there are modes, it holds for every azimuth, and the angle is right between the nodes too.

\emph{The model.} By Lemma~\ref{lem:linearization}, to first order the correction $u$ solves $\Delta u=0$ outside the hole
and $-\partial_ru+u/\varepsilon=\frac\kappa2\varepsilon\cos2\varphi$ on $|x|=\varepsilon$. The decaying solution is
\begin{equation}\label{eq:model}
u=\frac\kappa6\,\frac{\varepsilon^4}{|x|^2}\cos2\varphi,\qquad \max|u|=\frac{\kappa\varepsilon^2}6 .
\end{equation}

\begin{table}[ht]
\centering
\small
\setlength{\tabcolsep}{3.5pt}
\begin{tabular}{ccccccccc}
\toprule
$\varepsilon$ & $S_1$ & residual & $\max|H|$ & angle before & angle after & $\max|h-h_{\rm Costa}|$ & $\kappa\varepsilon^2/6$ & ratio\\[1pt]
\midrule
$0.02$ & $0.08$ & $9\cdot10^{-10}$ & $9\cdot10^{-8}$ & $1.00^\circ$ & $2\cdot10^{-11}$ & $1.164\cdot10^{-4}$ & $1.164\cdot10^{-4}$ & $1.000$\\
$0.05$ & $0.20$ & $6\cdot10^{-10}$ & $1\cdot10^{-6}$ & $2.50^\circ$ & $1\cdot10^{-10}$ & $7.288\cdot10^{-4}$ & $7.276\cdot10^{-4}$ & $1.002$\\
$0.10$ & $0.40$ & $5\cdot10^{-10}$ & $1\cdot10^{-7}$ & $4.97^\circ$ & $1\cdot10^{-10}$ & $2.933\cdot10^{-3}$ & $2.910\cdot10^{-3}$ & $1.008$\\
$0.20$ & $0.80$ & $6\cdot10^{-10}$ & $2\cdot10^{-4}$ & $9.74^\circ$ & $5\cdot10^{-10}$ & $1.197\cdot10^{-2}$ & $1.164\cdot10^{-2}$ & $1.028$\\
$0.30$ & $0.80$ & $5\cdot10^{-10}$ & $2\cdot10^{-4}$ & $14.13^\circ$ & $5\cdot10^{-10}$ & $2.717\cdot10^{-2}$ & $2.619\cdot10^{-2}$ & $1.037$\\
$0.45$ & $0.80$ & $1\cdot10^{-9}$ & $2\cdot10^{-4}$ & $19.55^\circ$ & $7\cdot10^{-10}$ & $5.778\cdot10^{-2}$ & $5.893\cdot10^{-2}$ & $0.980$\\
\bottomrule
\end{tabular}
\caption{The free-boundary collar on $65\times28$ nodes. ``Residual'' is the residual of the collocation system at the nodes;
``$\max|H|$'' is the largest mean curvature of the computed surface on the dense grid between the nodes. ``Angle before'' is the
largest deviation from $90^\circ$ of the angle between $M\setminus B_\varepsilon$ and the sphere; ``angle after'' is the same, in degrees, for the
computed surface, measured between the nodes. The last column compares the size of the correction with the
model~\eqref{eq:model}. All lengths are unscaled; after scaling by $\varepsilon^{-1}$ the correction at $\varepsilon=0.45$ is $0.13$.}
\label{tab:collar}
\end{table}

\begin{table}[ht]
\centering
\small
\begin{tabular}{cccccc}
\toprule
nodes & $21\times8$ & $29\times10$ & $37\times14$ & $45\times18$ & $65\times28$\\
\midrule
$\max|H|$ between the nodes & $0.85$ & $0.39$ & $0.078$ & $0.015$ & $2.3\cdot10^{-4}$\\
$\max|h-h_{\rm Costa}|$ & $0.057768$ & $0.057770$ & $0.057774$ & $0.057776$ & $0.057778$\\
\bottomrule
\end{tabular}
\caption{Convergence of the collar at $\varepsilon=0.45$, $S_1=0.8$ (unscaled lengths); $57\times24$ nodes give $1.3\cdot10^{-3}$.
The size of the correction is stable from the coarsest grid on; the mean curvature between the nodes is not, and it
needs the finer grids.}
\label{tab:convergence}
\end{table}

Tables~\ref{tab:collar} and~\ref{tab:convergence} and Figure~\ref{fig:collar} show the results. For $\varepsilon\leq0.1$ the size of the
correction agrees with the model~\eqref{eq:model} to $1\%$, and its profile converges to the model profile
(Figure~\ref{fig:collar}(b)). Up to $\varepsilon=0.45$ the deviation stays below $4\%$. On coarse grids the collocation residual is just
as small, but the mean curvature between the nodes is large: on $29\times10$ nodes it reaches $0.39$ (Table~\ref{tab:convergence}). On
the grid we use it is below $3\cdot10^{-4}$, and the angle at the boundary is right to $10^{-9}$ degrees. As a check of the discretization
we imposed the Costa heights at the hole instead of orthogonality; the solver then returns the Costa surface itself, up
to $1.5\cdot10^{-8}$. Moving $S_1$ from $0.8$ to $0.7$ changes the correction by $1.8\%$.

These computations are local: they do not solve the global problem of Section~\ref{sec:nonlinear}, whose solution is
determined by the whole torus, and the decay condition at $S_1$ is a model of the outside. They show that near the hole the
correction is the one predicted by the linear analysis, of relative size $\varepsilon$. For every $\varepsilon$ we tried, up to $0.45$, they give a
numerical approximation of a free-boundary minimal collar. This is not a proof that such a collar exists.

\subsection{The pictures}\label{sec:num_pictures}
Figure~\ref{fig:intro} shows a numerical model built from three pieces. The Costa surface is computed on a polar grid in $z$
around $\omega$ (Section~\ref{sec:num_impl}(a)), and every grid point is moved along its meridian.
\begin{itemize}
\item For $\varepsilon\leq|x|\leq S_1=0.8$ it is moved to the computed collar itself.
\item Beyond $S_1$ it is moved by the decaying tail $\sum_nb_n(S_1/|x|)^n\cos n\varphi$ of the correction, with $b_n$ the modal values at
$S_1$; the leading coefficient is $b_2=0.0258$ and $\sum_n|b_n|=0.034$, so the tail is at most $0.034\,(0.8/|x|)^2$. At $|x|=S_1$ the collar
and the tail agree to $2\cdot10^{-8}$, and their first derivatives agree by construction. The tail is \emph{not} a solution: the Costa
surface plus the tail has mean curvature up to $0.63$ at $|x|=0.8$, $0.37$ at $|x|=0.86$ and $0.26$ at $|x|=0.92$, decreasing outwards
(unscaled; after scaling by $\varepsilon^{-1}$ these values are multiplied by $\varepsilon=0.45$).
\item Inside the ball the collar is continued linearly; this part is cut away.
\end{itemize}
The surface is then scaled by $\varepsilon^{-1}$ and cut at the unit sphere. On the resulting mesh the angle with the sphere, computed
from finite differences, differs from $90^\circ$ by at most $0.05^\circ$ (against $19.5^\circ$ for the initial surface); on the smooth collar
it is right to $10^{-9}$ degrees. The boundary curve in Figure~\ref{fig:intro}(c) is computed directly from the collar. Since only the
modes $n\equiv2\pmod4$ occur, its height vanishes exactly at the azimuths $45^\circ+90^\circ k$. These are the points where the straight
lines $\{x_1=\pm x_2,\ x_3=0\}$ of the Costa surface meet the sphere.

Two caveats. First, Theorem~\ref{thm:A} gives $\Sigma_\varepsilon$ only for $\varepsilon<\varepsilon_0$, and we have no estimate for $\varepsilon_0$; for
$\varepsilon=0.45$ the existence of $\Sigma_\varepsilon$ is not proved. Second, we have no estimate for the distance between the model and an
exact $\Sigma_\varepsilon$; beyond the collar the model is the Costa surface with a small correction that is not minimal.

\subsection{The Costa--Hoffman--Meeks data}\label{sec:num_chm}
For $k=1$ the formula for $c_k$ in Section~\ref{sec:chm_data} gives $c_1=0.955977594972250$, which agrees with $\sqrt{2\pi/e_1}$ to $10^{-18}$.
For $k=1,\dots,5$ it gives $c_k=0.95598$, $0.98807$, $0.99512$, $0.99753$, $0.99858$. To check the period condition directly we
integrate $\Ree\Phi$ over the lifts of the circles of radius $\frac34$ around $\{0,1\}$ and $\{0,-1\}$; they close up after one turn, because the
monodromies of $w$ around $0$ and $\pm1$ multiply to $1$. With $c=c_k$ the real periods are below $10^{-15}$, $10^{-12}$ and $10^{-9}$ for
$k=1,2,3$, and with $c=1.05\,c_k$ they are about $0.13$. An auditing agent checked, in $85$-digit arithmetic and for $k=1,\dots,20$,
the local coordinate of Lemma~\ref{lem:chm_center}, the identities $\x\circ\lambda=S\x$, $\x\circ\kappa=T\x$ and the signs of the normals of
Lemma~\ref{lem:chm_symmetry}, to $10^{-85}$. For $k=1,\dots,9$ it checked that the boundary error is $-k\,a_k\varepsilon^k\cos((k+1)\varphi)$ up to a relative
correction of order $\varepsilon^{2k}$, as in Section~\ref{sec:chm}.

\emph{The pictures of Figure~\ref{fig:chm}.} The compact surface $\overline M_k$ is the union of two disks, each covering the $z$-sphere
$m$ times: the disk $|t|\leq1$ around the central point, with $z=-t^m$, and the disk $|s|\leq1$ around the planar end, with $z=s^{-m}$.
Their common boundary lies over $|z|=1$ and passes $m$ times through each catenoidal end. We integrate $\Phi$ on polar grids of
both disks and match them at one boundary point. Near the catenoidal ends we use separate polar grids in the end
coordinate $\zeta$, $z\mp1=\zeta^m$, which cover a whole neighbourhood of each end. Where the grids overlap they agree to $5\cdot10^{-7}$ for
$k=2$ and to $4\cdot10^{-6}$ for $k=3$. The collar is computed as in Section~\ref{sec:num_collar}, with the modes $(2j+1)m$, $\varepsilon=0.56$, $S_1=1.0$
for $k=2$ and $\varepsilon=0.72$, $S_1=1.3$ for $k=3$; at these values the angle of $M_k\setminus B_\varepsilon$ with the sphere is off by $10.9^\circ$ and
$4.1^\circ$. On $65\times28$ nodes the residual is below $10^{-9}$. Between the nodes the mean curvature is at most $4\cdot10^{-8}$ and
$9\cdot10^{-9}$, and the angle is right to $10^{-9}$ degrees. The size of the correction is $2.91\cdot10^{-2}$ and $1.04\cdot10^{-2}$, against
$\frac{ka_k}{k+2}\varepsilon^{k+1}=2.70\cdot10^{-2}$ and $1.03\cdot10^{-2}$ from the model of Section~\ref{sec:chm}. The rest of the surfaces, which come
closest to the origin at $|x|=1.54$ for $k=2$ and $2.03$ for $k=3$, stays outside the collar.

\subsection{Reproducibility}\label{sec:num_scripts}
All scripts are attached to the arXiv submission as ancillary files, in the directory \texttt{anc/scripts/}; run from there,
they write their output to \texttt{anc/figures/}. The numerical part needs Python~3 with NumPy, SciPy and
mpmath; the pictures also need Matplotlib and PyVista~0.49 (VTK~9.7). The interactive models at \viewer{costa} are made from the
same data.
\begin{itemize}
\item \texttt{costa\_numerics.py} (NumPy and SciPy only): the Weierstrass data, the exact level points, the collar solver and
its verification between the nodes.
\item \texttt{costa\_paper\_figures.py}: Sections~\ref{sec:num_impl}--\ref{sec:num_collar}, Tables~\ref{tab:defect}--\ref{tab:convergence},
Figures~\ref{fig:costa}, \ref{fig:torus} and~\ref{fig:defect}--\ref{fig:collar}; it writes all numbers to the file \texttt{costa\_numbers.txt} in \texttt{figures/costa/}.
\item \texttt{visualize\_exterior\_fbms.py}: Figure~\ref{fig:intro}, made by the functions \texttt{fig\_costa} and
\texttt{fig\_costa\_boundary}.
\item \texttt{costa\_theta\_core.py} and \texttt{costa\_jacobi\_spectrum.py N}: the theta-series implementation and the spectrum
of Section~\ref{sec:num_spectrum}.
\item \texttt{costa\_symmetry\_check.py}: Section~\ref{sec:num_symmetry}.
\item \texttt{costa\_ends\_check.py} (SymPy): the symbolic check of Lemma~\ref{lem:ends}(iii).
\item \texttt{chm\_period\_check.py} and \texttt{chm\_checks.py} (the latter written by an auditing agent): Section~\ref{sec:num_chm};
the class \texttt{CHM} in \texttt{costa\_numerics.py} and the function \texttt{fig\_chm} in \texttt{visualize\_exterior\_fbms.py}:
Figure~\ref{fig:chm}.
\item \texttt{costa\_hole\_experiment.py}: the first version of the experiment, with interpolated Costa data (superseded by
\texttt{costa\_numerics.py}).
\end{itemize}

\section*{Tool and computational resource disclosure}
The construction, the proofs and the numerical experiments in this paper were developed with substantial assistance from AI
systems. Anthropic's Claude models proposed the construction, carried out the numerical experiments, made the pictures and
drafted the text. The proofs in Sections~\ref{sec:costa}--\ref{sec:nonlinear} are those of an earlier draft, which was
audited by five independent Claude agents (Sections~2, 3, 4, 5 and a global adversarial review with a literature search).
Their findings are incorporated: one major (the estimate at the ends, now done with a vertical field) and about twenty
minor. One of the auditing agents wrote the theta-series implementation and the spectral computation of
Section~\ref{sec:num_spectrum}. Two further independent audits (30 September and 1 October 2026) are also incorporated. From
them come the corrected weight in Section~\ref{sec:linear}, the domain in Theorem~\ref{thm:A}(iii), the argument with the Gauss map in
Lemma~\ref{lem:no_kernel}, the numbering of the cited results of Mazet and Mendes, the explicit formulas of
Lemma~\ref{lem:ends} and Proposition~\ref{prop:index}. The second audit found that the first version of the collar computation had
large mean curvature between the collocation nodes, and Section~\ref{sec:num_collar} was redone after it. The extension to the
Costa--Hoffman--Meeks surfaces in Section~\ref{sec:chm} was worked out by the agent of the first of these audits. Its
ingredients were then checked independently: the period constant, the expansion at the central point, the symmetries and
the modes (Section~\ref{sec:num_chm}). All statements and proofs have been checked by the author, who takes full responsibility for
them and for any remaining errors. The manuscript has not yet been verified by an independent human expert.

{\raggedright

}

\end{document}